\documentclass[aip,reprint,graphicx,amsmath,amssymb]{revtex4-2}
\usepackage[usenames,dvipsnames]{xcolor}
\usepackage{appendix}
 \usepackage{amsthm}
 \usepackage{graphicx}
 \usepackage{booktabs}
 \usepackage{float}

\usepackage{slashbox}
\usepackage{comment}
\usepackage{soul}
\usepackage{comment}
\usepackage[hidelinks]{hyperref}
\usepackage[utf8]{inputenc}

 \newtheorem{Proposition}{Proposition}

\draft 

\begin{document}


\title{A mathematical model of CAR-T cell therapy in combination with chemotherapy for malignant gliomas} 
\author{Dmitry Sinelshchikov}
\email{dmitry.sinelshchikov@ehu.eus.}
\affiliation{Instituto Biofisika (UPV/EHU, CSIC), University of the Basque Country, Leioa, 48940, Spain}
 \affiliation{Ikerbasque, Basque Foundation for Science, Bilbao 48009, Spain}
\author{Juan Belmonte-Beitia}%
 \email{juan.belmonte@uclm.es.}
 \author{Matteo Italia}
 \email{matteo.italia@uclm.es.}
\affiliation{
Mathematical Oncology Laboratory (MOLAB), Departament of Mathematics, Instituto de Matem\'atica Aplicada a la Ciencia y la Ingenier\'ia,\\ Escuela T\'ecnica Superior de Ingenier\'ia Industrial,
Universidad de Castilla-La Mancha.\\
Laboratorio de Oncolog\'ia Matem\'atica, Instituto de Investigaci\'on Sanitaria de Castilla-La Mancha (IDISCAM),\\
Ciudad Real 13071, Spain.
}%


\begin{abstract}
{
Malignant gliomas (MG) are among the most aggressive primary brain tumors, characterized by a high degree of resistance to therapy and poor prognosis. In this work, we develop a mathematical model to investigate the dynamics of MG under the combined effects of chemotherapy and chimeric antigen receptor (CAR-T) cell therapy. The proposed model is a five-dimensional dynamical system incorporating impulsive inputs that correspond to the clinical administration of chemotherapy and immunotherapy. We demonstrate the non-negativity of solutions for non-negative initial conditions, ensuring the biological relevance of the model.
\\
We show that if we apply both therapies only once, the trajectories are attracted to an invariant surface corresponding to the tumor carrying capacity.
Conversely, under constant administration of both treatments, we identify parameter ranges in which tumor eradication is achievable.
Furthermore, we numerically study various treatment combinations to determine optimal protocols at the population level.
To this end, we generate a cohort of $10^4$ virtual patients with model parameters sampled uniformly within clinically relevant ranges and carry out \textit{in silico} trials. Our findings indicate that tumor growth rate, chemotherapy efficacy, and tumor-induced immunosuppression are the key determinants of survival outcomes.
\\
We believe that our results provide new theoretical insights into treatment optimization and offer a framework for refining the design of clinical trials for MG therapies.}


\end{abstract}


\maketitle 

\begin{quotation}
Malignant gliomas (MG), the most common primary brain tumors, are characterized by high invasiveness and resistance to therapy, with a median survival of less than 15 months despite standard multimodal treatments like the Stupp protocol. Novel Chimeric Antigen Receptor (CAR)-T cell therapies have shown revolutionary potential in oncology but face significant challenges in solid tumors like MG, including antigen heterogeneity, immunosuppressive tumor microenvironment, and blood-brain barrier limitations. Recent advances, such as CARs targeting multiple antigens, aim to overcome these obstacles.

Chemotherapy, such as temozolomide (TMZ), has opposite effects on CAR-T cells: while it directly kills CAR-T cells, it can also enhance their efficacy. Chemotherapy can upregulate tumor antigen expression, improving CAR-T cell targeting, and lymphodepleting regimens can create a more favorable immune microenvironment to support CAR-T cell activity. Preclinical studies and ongoing clinical trials suggest that this combination may provide synergistic benefits in MG treatment. To explore this, we develop a mathematical model to analyze the dynamics of combined therapy and identify optimal treatment strategies using analytical and numerical methods.

Our mathematical framework, under constant treatment for both therapies, identifies feasible critical thresholds for TMZ and CAR-T cell doses that could lead to tumor eradication. {Due to} patient toxicological constraints, we then simulate several clinically feasible therapeutic protocols with {a} {finite number of} administrations. We first study the applications of TMZ and CAR-T cells as monotherapy. Our model highlights the critical role of tumor immunosuppression, which can limit the efficiency of CAR-T cells. In general, our results are in agreement with clinical data, demonstrating the robustness and predictive value of our model.

We also
explore various combinations of TMZ and CAR-T cell treatments to identify the most effective therapeutic strategy.  We compare outcomes across protocols and investigate the impact of the parameters on the survival of virtual patients. Our results show the key roles of tumor proliferation, TMZ killing efficacy, and tumor immunosuppression as survival prognostic biomarkers.
Notably, an alternating schedule of CAR-T cell therapy and TMZ yields the highest overall survival in \emph{in silico} trials, with a median survival of nearly 650 days.
Our results provide theoretical insights to guide the design of clinical trials for MG therapies.
\end{quotation}

\section{Introduction}

Gliomas are the most common type of primary brain tumor, with glioblastomas (GBM) representing the most aggressive and therapeutically challenging subtype due to their highly infiltrative nature, which prevents complete eradication by current treatments \cite{gliomas2024}. Patients with gliomas typically succumb to complications arising from tumor progression. The current standard of care for aggressive MG such as GBM, known as the Stupp protocol \cite{stupp_protocl2005}, combines surgical resection with chemoradiotherapy using temozolomide (TMZ) as chemotherapy. Despite this multimodal approach, the prognosis remains poor, with a median overall survival of less than 15 months \cite{stupp_protocl2005}.

To address the fatal prognosis associated with gliomas, new therapies and combinations of innovative and traditional approaches are under investigation. Among these, Chimeric Antigen Receptor (CAR)-T cell therapy has emerged as one of the most promising and revolutionary cancer treatments in recent decades. For recent reviews on the applications of CAR-T therapy to gliomas, see \cite{cart_review_MG,review_cart_diffuse_midline_glioma,montoya2024roadmapCART,neurooncoReview2024}.

CAR-T cells are genetically modified T cells, derived either autologously (from the patient) or allogeneically (from a donor). Their extracellular domain is engineered to recognize specific tumor-associated antigens, while the intracellular domain contains signaling elements that trigger T-cell activation. Upon CAR engagement with the associated antigen, primary T-cell activation occurs, leading to cytokine release, cytolytic degranulation, target cell death, and T-cell proliferation \cite{Feins}.

CAR-T cells targeting CD19+ cells have shown remarkable success in treating B-cell malignancies, particularly in patients with acute lymphoblastic leukemia \cite{Maude,Miliotou}. Similarly, positive outcomes have been achieved in multiple myeloma \cite{Dagostino}, diffuse large B-cell lymphoma \cite{Chavez}, and refractory acute myeloid leukemia using CD33-specific CAR-T cells \cite{Wang}. These encouraging results lead to extensive research on the application of CAR-T therapies to solid tumors \cite{Moon}, with ongoing clinical trials investigating their potential in glioblastomas, as well as gastrointestinal, genitourinary, breast, and lung cancers, among others \cite{Bagley}.
Despite this progress, CAR-T therapy for solid tumors
 continues to face significant challenges.
The primary challenge lies in identifying tumor-specific antigens that are expressed exclusively in cancer cells to minimize on-target off-tumor toxicity \cite{Watanabe,kringel2023chimeric}.
Another critical consideration is ensuring that the antigens selected for therapy are humanized to prevent the development of neutralizing antibodies against CAR-T cells \cite{Hege}. Beyond antigen targeting, several additional challenges remain, including the persistence and expansion of CAR-T cells, their ability to effectively infiltrate and traffic
within
tumors, and the impact of tumor-driven immune resistance mechanisms, all of which can significantly
compromise CAR-T cell efficacy \cite{Ma,kringel2023chimeric}.

For these reasons, it is crucial to develop strategies to improve and optimize the effectiveness of CAR-T cell therapy \cite{Hong}.
This task is particularly challenging in the context of brain tumors, which are highly heterogeneous: tumor cells do not express the same antigens uniformly \cite{Bodnar2023amcs}.
Such heterogeneity limits the efficacy of CAR-T therapies, as not all tumor cells may be targeted. To address this limitation,
novel dual-specific tandem CAR-T cells have shown promising \textit{in vitro} efficacy in overcoming antigen specificity challenges \cite{tandemCART}. These CAR-T cells are engineered to simultaneously target two well-characterized glioblastoma antigens, EGFRvIII and IL-13R$\alpha$2, which are frequently expressed in GBM cells but completely absent in normal brain tissue. Moreover, trivalent CAR-T cells capable of targeting three antigens simultaneously offer promising solutions to this issue, potentially covering a broader spectrum of GBM cells \cite{trivalentCART}.

A promising approach involves combining CAR-T therapy with chemotherapy.
This strategy could leverage chemotherapy to eliminate tumor cells that do not express the antigens required for CAR-T cell recognition and attack. Chemotherapy has been shown to enhance the expression of tumor antigens in solid tumors, thereby amplifying the immune response. This effect has been demonstrated in GBM models using NKG2D CAR-T cells, where chemotherapy improved CAR-T cell efficacy by increasing target antigen availability \cite{chemo_increase_antigen}. Subsequently, CAR-T cells could be administered to target residual antigen-expressing tumor cells, providing a more comprehensive and effective treatment approach.
In addition, a key problem with chemotherapy is the development of drug resistance, often leading to treatment failure. In this context, combining
CAR-T cell therapy with chemotherapy could help circumvent this problem, as CAR-T cells would target and eliminate chemotherapy-resistant cells. Furthermore, gliomas present significant barriers to effective therapy, including immunosuppressive tumor microenvironments and the restrictive blood-brain barrier \cite{kringel2023chimeric}. Lymphodepleting chemotherapy, which
reduces
the host's immune cell population to create space for CAR-T cells and other therapeutic agents, is under investigation as a means to overcome these challenges \cite{lymphodepleting_chemio}. By weakening immune suppression and facilitating  CAR-T cell trafficking to the tumor site, this approach has the potential to improve therapeutic efficacy.

Although preclinical studies have shown the potential of such combination approaches in animal models \cite{suryadevara2018temozolomide,combined_CART_TMZ},
their effectiveness in clinical practice is
currently under investigation in an ongoing trial (NCT04165941) \cite{TMZ_CART_trial}.
At this stage, mathematical models can help describe, understand, and predict the effects of this combined therapy on tumor progression \cite{Altrock,italia2025silico}.
\textit{In silico} trials use virtual simulations to model such experiments, offering a powerful tool in mathematical oncology \cite{gevertz2024assessing, wang2024virtual}. These computational approaches support clinical trials by reducing the risk of failure through mechanistic and predictive simulations \cite{brown2022derisking}.

Mathematical approaches are always playing a more crucial role in medicine, providing powerful tools to unravel the complexities of diseases. By formulating models and analyzing data, researchers have gained deeper insights into disease mechanisms \cite{italia2022calibrated}, optimal \cite{italia2022optimal} and personalized \cite{italia2023mathematical} treatments, and effective healthcare management \cite{italia2023model}. In oncology,
mathematical models, which serve to describe, quantify and predict complex behaviors, have significant potential to optimize administration protocols, deepen understanding of therapeutic dynamics, and help design clinical trials \cite{Altrock}.
These models can unravel intricate systems, such as those involving interactions between immune systems, tumor cells, healthy cells, and treatments, and have been widely used for this purpose in recent years \cite{Jose1,italia2022ecc,Jose2}.
Recent mathematical modeling studies have investigated various aspects of CAR-T cell therapies \cite{Sahoo, kimmel21,barros20,Marciniak,Ode1,Victor,SSerrano,sabir2025mathematical}.

In this paper, we employ an ordinary differential equation mathematical model to study the response of MG to the combination of CAR-T therapy and chemotherapy. As chemotherapeutic agent, we take temozolamide (TMZ), which is the standard chemotherapeutic agent for this kind of tumors \cite{friedman2000temozolomide,stupp_protocl2005}. We make a theoretical study of the model, showing the existence of steady states, stability and positivity of solutions, and calculating invariant surfaces of the mathematical model. We consider two kinds of treatments, constant and with a finite number of applications, and we study the insights of these treatments from the mathematical and medical points of view, performing numerical simulations.

\begin{figure}[h!]
    \centering
    \includegraphics[width=0.95\linewidth]{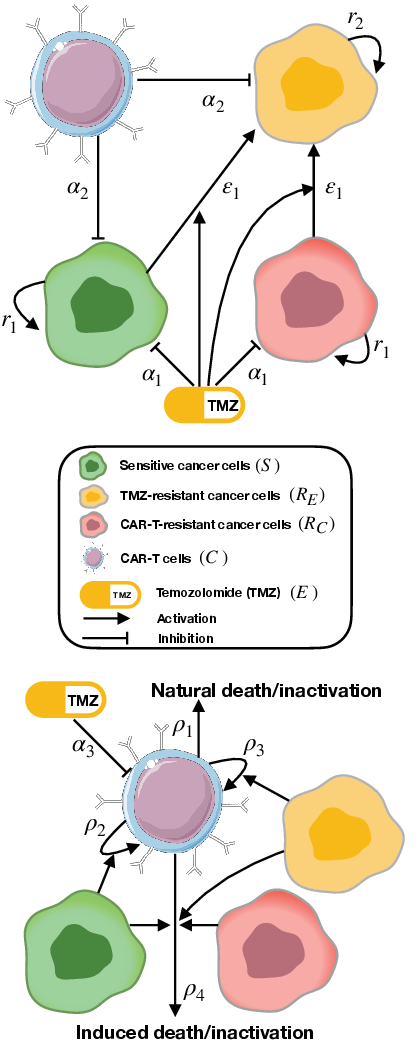}

    \caption{Graphical representation of the model variables and their interactions, focusing on the cancer cells dynamics and how they are affected by the treatments (top), and on the CAR-T cells dynamics and how it is regulated by cancer cells and TMZ (bottom).}
    \label{fig:model}
\end{figure}

The paper is organized as follows. In Section~\ref{sec:model}, we introduce the biological background and present the mathematical model, explaining its variables and parameters.
We also provide a brief background on the mathematical analysis that will be used throughout the manuscript, along with the definition of a virtual patient and \emph{in silico} trials to give the reader context.
Next, Section~\ref{sec:result} presents the results in two parts. The first part details the theoretical results obtained from the mathematical analysis carried out, using the qualitative theory of differential equations and analyzing the asymptotic dynamics of our model. In the second part, we present numerical results for the applications of TMZ monotherapy, CAR-T cell monotherapy, and finally, the combined treatment of CAR-T and TMZ. Lastly, Section~\ref{sec:conclusions} discusses the work, its results, and the insights gained.

\section{Materials and methods}\label{sec:model}

In this Section, we first present the biological knowledge underlying the mathematical model in Section \ref{sec:bio_back}, followed by a detailed description of the model equations, parameters, and simulation settings in Section \ref{subsec:model}. Next, Section \ref{sec:math_back} outlines the mathematical background {required for the subsequent} analysis. Finally, in Section \ref{sec:vp_insilico_trials}, we introduce the frameworks for virtual patients and \textit{in silico} trials.

\subsection{Biological background}\label{sec:bio_back}

In this study, we develop a mathematical model of MG cells, categorizing them as either sensitive or resistant to TMZ and CAR-T cell therapy. We construct an {ODE-based} model to describe the growth of MG and its response to combined TMZ and CAR-T cell therapies. The model is based on a compartmental framework that captures the temporal evolution of several homogeneous tumor cell subpopulations and their responses to TMZ and CAR-T. Specifically, it includes subpopulations of sensitive ($S$), resistant to CAR-T cell therapy but sensitive to TMZ ($R_C$), and resistant to TMZ but sensitive to CAR-T ($R_E$) cells. For simplicity, spatial dependencies and stochastic effects are not considered in this model. In the following, we address the biological background on which the model is based.

We assume that sensitive cells ($S$) have a proliferative phenotype
and are sensitive to both chemotherapy \cite{chemodamage} and CAR-T therapy \cite{Feins}. Note that CAR-T-sensitive cells express the antigen(s) recognized by CAR-T cells, enabling targeted destruction, while resistant cells lack this expression and evade immune-mediated killing \cite{Bodnar2023amcs}. Some CAR-T cell antigens are related to MG growth and TMZ resistance \cite{tandemCART,trivalentCART} and cells that do not express these antigens could grow slower and be more sensitive to TMZ. However, for the sake of simplicity, we assume that CAR-T-resistant cells ($R_C$) exhibit the same proliferative phenotype and TMZ sensitivity as sensitive cells ($S$).

TMZ resistance is known to be induced in TMZ sensitive cells ($S$ and $R_C$) upon exposure to TMZ \cite{rabe2020identification}. Indeed, the model incorporates sensitivity and resistance to TMZ, the standard chemotherapeutic agent for gliomas \cite{friedman2000temozolomide,stupp_protocl2005}. Over time, a subset of tumor cells may remain sensitive and responsive to treatment, while others acquire resistance, reflecting the complex and dynamic nature of tumor heterogeneity under dual therapy \cite{rabe2020identification}. Different studies report TMZ-resistant cells proliferate at a slower, similar, and faster rate than sensitive cells \cite{campos2014aberrant,stepanenko2016temozolomide,yuan2018abt,gupta2014discordant,dai2018scd1,Delobel_PLOS}. Thus, we assume that TMZ-resistant cells ($R_E$) have a proliferative phenotype that may differ from TMZ-sensitive cells. We assume that TMZ-resistant cells express the antigen(s) recognized by CAR-T \cite{chemo_increase_antigen}, thus, these cells are sensitive (only) to CAR-T therapy and we do not include populations resistant to both TMZ and CAR-T therapy.

We also consider the dynamics of the applied treatments. TMZ decays exponentially between consecutive administrations following first-order pharmacokinetics \cite{agarwala2000temozolomide}, it acts only on TMZ-sensitive tumor cells and can induce resistance in these cells \cite{rabe2020identification}. Instead, CAR-T cells compete with tumor cells, become naturally inactivated, and attack only the tumor population in which tumor antigen(s) is expressed \cite{Brown2,Fonkoua}.

\subsection{The model}\label{subsec:model}

Here, we present the mathematical equations underlying the model, derived from biological assumptions. We also define the model parameters, initial and final conditions used in the numerical simulations, and describe how TMZ and CAR-T therapies are implemented within the simulation framework.

\subsubsection{Model equations}

According to what was previously mentioned, it follows that the variables involved in the model are:

\begin{itemize}
\item $S$ is the subpopulation of tumor cells sensitive to both CAR-T cell and TMZ treatment (unit: cells);
\item $R_C$ is the subpopulation of tumor cells resistant to the CAR-T treatment but sensitive to TMZ (unit: cells);
\item $R_E$ is the subpopulation of tumor cells resistant to TMZ but sensitive to CAR-T (unit: cells);
\item $C$ is the compartment of CAR-T cells (unit: cells);
\item $E$ is the normalized concentration of the drug, i.e., the efficacy of the TMZ (unit: {dimensionless});
\item $t$ represents the time (unit: days).
\end{itemize}

Figure~\ref{fig:model} provides a graphical representation of the model, illustrating the interactions among variables. The top panel focuses on the dynamics of cancer cells and the effects of treatments, while the bottom panel depicts the dynamics of CAR-T cells.

Then, the model is formulated as follows.
\begin{widetext}
\begin{eqnarray}
\label{eq:new_model}
\frac{dS}{dt}&=&r_1 S \left(1-\frac{S+R_{C}+R_{E}}{K}\right)-\alpha_1 E S -\varepsilon_{1} E S-\alpha_2 CS, \label{eqS} \\
\frac{dR_C}{dt}&=&r_1R_C\left(1-\frac{S+R_{C}+R_{E}}{K}\right)-\alpha_1E R_C-\varepsilon_{1} E R_C, \label{eqRC}  \\
\frac{dR_E}{dt}&=&r_2R_E\left(1-\frac{S+R_{C}+R_{E}}{K}\right)-\alpha_2CR_E+\varepsilon_1 (S+R_C)E, \label{eqRE} \\
\frac{dC}{dt}&=&-\rho_1C+\frac{\rho_2SC}{g_1+S} +\frac{\rho_3R_EC}{g_2+R_E}-\rho_{4}\frac{(S+R_{C}+R_{E})C}{g_{3}+C}-\alpha_3EC \label{eqC},\\
\frac{dE}{dt}&=&-\mu E \label{eqE}.
\end{eqnarray}
\end{widetext}

Equation \eqref{eqS} describes the dynamics of the sensitive tumor cell subpopulation ($S$). In the absence of treatments, these cells grow following a logistic model characterized by a {proliferation rate $r_1$} and a carrying capacity $K$ \cite{Gerlee,Budia}. TMZ kills sensitive cells at a rate $\alpha_1$, while simultaneously inducing resistance at a rate $\varepsilon_1$ \cite{Delobel_PLOS}. {This implies that, at a rate $\varepsilon_1$, sensitive tumor cells ($S$) interacting with TMZ ($E$) acquire resistance to the drug. Consequently, the term $\varepsilon_1 ES$ is removed from the first equation and incorporated into Equation \eqref{eqRE}, which governs the compartment of TMZ-resistant cells. While an intermediate transient step in the development of TMZ resistance was considered in \cite{Delobel_PLOS}, we adopt a simplified approach using the terms described above.}
Additionally, CAR-T cells target and eliminate sensitive cells at a rate $\alpha_2$ \cite{Ode2}, further modulating population dynamics.

\begin{table*}[]
\caption{Descriptions, values, and sources of the reference parameters for system  \eqref{eqS}--\eqref{eqE}.}
 \label{tab:paramenters}
\begin{center}
\begin{tabular}{@{}clllc@{}}
Parameter & Description & Reference value & Unit & References \\
\hline
$\rho_1$ & activated CAR-T cell mean lifetime in the tumor site & 1/7-1/30 & day$^{-1}$ & [\onlinecite{Ghorashian}]\\
${\rho}_2$ & mitotic stimulation of CAR-T cells by sensitive tumor cells & 0.2-0.9 & day${^{-1}}$ & [\onlinecite{Ode2}]\\
${\rho}_3$ & mitotic stimulation of CAR-T cells by {TMZ-}resistant tumor cells & 0.2-0.9 & day${^{-1}}$ &[\onlinecite{Ode2}]\\
${\rho}_4$ & tumor inactivation rate & 0.01-0.2 & day${^{-1}}$ &[\onlinecite{santurio2022mathematical}]\\
$r_1$ & sensitive tumor growth rate & 0.001-0.025 & day${^{-1}}$ & [\onlinecite{Ode2}]\\
$r_2$ & TMZ-resistant tumor growth rate & 0.0005-0.05 & day${^{-1}}$ & assumed\\
$g_1$ & T cells
for half-maximal CAR-T cell proliferation {due to $S$} & $1\times 10^{10}$ & cell & [\onlinecite{Ode2}]\\
{$g_2$} & T cells
for half-maximal CAR-T cell proliferation {due to $R_E$}& $1\times 10^{10}$ & cell & [\onlinecite{Ode2}]\\
{$g_3$} & CAR-T concentration for half-maximal tumor inactivation &  $2\times 10^9$ & cell & [\onlinecite{Ode2}]\\
$\alpha_1$ & Chemotherapy killing efficiency against tumor & 0.1-1 & day${^{-1}}$ & [\onlinecite{Delobel_PLOS}]\\
${\alpha}_2$ & CAR-T cells killing efficiency against tumor & $2.5\times 10^{-10}$ &
day${^{-1}}$cell${^{-1}}$ & [\onlinecite{Ode2}]\\
${\alpha}_3$ & Chemotherapy killing efficiency  against CAR-T & 0.1-1 & day${^{-1}}$ & assumed\\
${\epsilon}_1$ & transition rate from sensitive to TMZ-resistant cancer cells & $0.1-0.6$ & day${^{-1}}$ & [\onlinecite{Delobel_PLOS}]\\
${K}$ & tumor carrying capacity& $5\cdot 10^{12}$ & cell & [\onlinecite{Forys}] \\
$v$ & CAR-T cells dosage & $10^{7}-10^9$ & cells  & [\onlinecite{goff2019pilot}]\\
{$E_0$ }& TMZ normalized dosage & $0-1$ & -  & [\onlinecite{Delobel_PLOS}]\\
$\mu$ & TMZ clearance rate & $8.32$   &  day${^{-1}}$ & [\onlinecite{Delobel_PLOS}] \\
$\delta_{1}$ & initial fraction of resistant to TMZ cells & $0.0001 - 0.1$   &  - & assumed \\
$\delta_{2}$ & initial fraction of resistant to CAR-T cells & $0.1 - 0.5$   &  - & [\onlinecite{o2017single}]  \\
{$L_{1}$} & number of applied TMZ cycles & 10   &  - & [\onlinecite{stupp_protocl2005}] \\
{$L_{2}$ }& number of CAR-T cell injections & 2   &  - & assumed  \\
\end{tabular}
\end{center}
\end{table*}

Equation \eqref{eqRC} models the dynamics of CAR-T-resistant cells ($R_C$), following the same structure as Equation \eqref{eqS}, with the key difference that these tumor cells do not express the antigen(s) and are therefore unaffected by CAR-T cells (the term involving $\alpha_2$ is absent).
Note that the same biological resistance mechanism described in Equation \eqref{eqS} is also present here. Specifically, a fraction of $R_C$ cells in contact with TMZ acquires resistance at the same rate $\varepsilon_1$. Consequently, the term $\varepsilon_1 ER_C$ appears both here and in Equation \eqref{eqRE}.

Equation \eqref{eqRE} describes the dynamics of tumor cells resistant to the efficacy ($E$) of TMZ ($R_E$). In the absence of treatment, these cells ($R_E$) follow logistic growth, similar to sensitive cells ($S$), with the same carrying capacity $K$ and a proliferation rate $r_2$, which may differ from that of sensitive cells ($r_1$). CAR-T cells also target and eliminate these resistant cells at the rate $\alpha_2$, as $R_E$ cells express the same target antigen(s) as $S$ cells.
The final term represents the conversion of TMZ-sensitive cells ($S$ and $R_C$) into TMZ-resistant cells ($R_E$) due to TMZ ($E$), governed by the rate $\varepsilon_1$.

Equation \eqref{eqC} models the dynamics of CAR-T cells. The first term represents the natural death or inactivation of activated CAR-T cells, occurring at a rate $\rho_{1}$. The second and third terms describe the proliferation of CAR-T cells upon encountering CAR-T sensitive tumor cells ($S$ and $R_E$), with rate constants $\rho_2$ and $\rho_3$ and saturation levels $g_1$ and $g_2$, respectively. The fourth term accounts for CAR-T cell inactivation due to interactions with tumor cells and their microenvironment, with a rate $\rho_4$ per tumor cell and a saturation threshold of $g_3$ CAR-T cells. Finally, the last term captures the destruction of CAR-T cells by TMZ at a rate $\alpha_{3}$. Similar equations describing CAR-T cell dynamics have been employed in \cite{Ode2,Bodnar2023amcs,Forys}.

Lastly, Equation \eqref{eqE} is a first-order kinetics describing the normalized concentration of the drug as in \cite{Derippe22}, i.e., the efficacy of the drug ($E$), which decays exponentially with a constant rate $\mu$.

\subsubsection{Parameter values}

We have carried out a comprehensive review of the literature and experimental data to obtain biologically relevant values for the model parameters. Table \ref{tab:paramenters} summarizes the parameters, detailing their biological meanings, assigned values, and corresponding sources. When no source is available, we assume some ranges based on biological and rational reasoning explained, as explained below.

When MG cells become resistant to TMZ, their cancer growth rate ($r_2$) can vary depending on multiple factors, including the underlying resistance mechanisms and the tumor microenvironment. In this study, we assume that TMZ resistance is achieved through energy-dependent mechanisms, leading to a reduced growth rate, as suggested by several studies (e.g., references \cite{campos2014aberrant,stepanenko2016temozolomide,yuan2018abt,dai2018scd1}). Specifically, we set $r_2 = r_1 / 2$ for each virtual patient throughout the manuscript. However, previous reports have indicated that TMZ-resistant cells may exhibit growth rates similar to or higher than those of TMZ-sensitive cells \cite{gupta2014discordant,stepanenko2016temozolomide,dai2018scd1,Delobel_PLOS}. To address these scenarios, we relax this assumption and analyze the implications in Sections ~\ref{sec:res_TMZ_only} and ~\ref{sec:res_fast_grow}.

To the best of our knowledge, no experiments have directly evaluated TMZ's cytotoxic efficacy against CAR-T cells in MG patients. However, since CAR-T cells share the characteristics of actively proliferating cells with MG and TMZ kills actively proliferating cells, we assume that the killing efficacy of TMZ against CAR-T cells ($\alpha_3$) is comparable to its efficacy against tumor cells ($\alpha_1$). For simplicity, we adopt $\alpha_3 = \alpha_1$ for each virtual patient throughout the manuscript.

Different mechanisms are responsible for the intrinsic (inherent ability of MG to resist TMZ) and acquired (develops during or after treatment) TMZ resistance in MG \cite{tomar2021elucidating}. The primary mechanism of intrinsic resistance is the expression of the O6-Methylguanine-DNA Methyltransferase (MGMT) gene, which encodes a DNA repair enzyme. In fact,
high levels of MGMT expression are associated with poor response to TMZ \cite{tomar2021elucidating}. However, the initial fraction of TMZ-resistant cells in MG ({$\delta_1$}) has not been clearly characterized. Here, we assume $\delta_1$ in $[10^{-4},0.1]$ to always account for intrinsic resistance, while maintaining at the same time a predominance of sensitive cells before treatment.

\subsubsection{Initial and final conditions}

During computations, we initially assume that there are $T^{(0)}$ tumor cells in total, $S^{(0)}+R_{C}^{(0)}+R_{E}^{(0)}=T^{(0)}$. We also set the fractions of resistant cells $R_{E}^{(0)}$ and $R_{C}^{(0)}$, which we denote $\delta_{1}$ and $\delta_{2}$, respectively.  Initial conditions for CAR-T cells ($C$) and for efficacy of TMZ ($E$) are zero in our simulations. Thus, initial conditions for our model are given by
\begin{eqnarray*}
    S(0)&=&S^{(0)}=T^{(0)}(1-\delta_{1}-\delta_{2}),\\R_E(0)&=&R_{E}^{(0)}={\delta_{1}}T^{(0)}, \\ R_C(0)&=&R_{C}^{(0)}={\delta_{2}}T^{(0)}, \\
C(0)&=&0,\quad \quad E(0)=0.
\end{eqnarray*}

Note that the initial conditions for CAR-T cells ($C$) and TMZ ($E$) can be positive when describing situations where treatments were previously applied.

In our numerical simulations, the final condition corresponds to patient death, which occurs when the tumor reaches the fatal volume of $K/5$, equivalent to $10^{12}$ cells. Although there is no exact number for the number of tumor cells in a glioma that leads to death, $10^{12}$ could correspond approximately to 100 grams or a volume of $300cm^3$, which has been reported as a fatal tumor volume \cite{Delobel_PLOS}.

\subsubsection{Modeling treatment applications}

The administration of CAR-T cells and TMZ follows discrete treatment protocols, with specific dosing schedules that significantly influence system dynamics. These treatments are modeled as instantaneous events due to the nature of CAR-T injections and the rapid pharmacokinetics of TMZ, where peak plasma concentrations are typically reached within 30–90 minutes post-administration \cite{neuropharmacokinetics}, a timeframe much shorter than tumor growth dynamics. The application of treatments introduces discontinuities in the treatment state variables \cite{Delobel_PLOS,italia2022optimal,italia2023mathematical}. Specifically, the time integration of the system is halted precisely at the moment of treatment administration. The affected variables (CAR-T cell concentration $C$ and TMZ efficacy $E$) are then updated to reflect the treatment dosage, and integration resumes with the modified initial conditions. This approach captures the inherently discontinuous nature of therapeutic interventions within a continuous ODE framework.

Let $T_{1}$ denote a time point when TMZ is administered, and $T_{2}$ a time point when CAR-T cells are injected.

Then
\begin{gather}
E(T_{1}+0)=E(T_{1}-0)+E_{0}, \label{doses_eq1}\\
C(T_{2}+0)=C(T_{2}-0)+v,
\label{doses_eq2}
\end{gather}
where $E_{0}$ denotes the normalized amount of administered TMZ ($0\leq E_{0}\leq1$), and $v$ represents the number of CAR-T cells injected. In equation \eqref{doses_eq1} and \eqref{doses_eq2} and throughout the manuscript, the expressions $\mp 0$ denote standard one-sided limits, which biologically represent the time immediately before and just after the application of therapy (TMZ or CAR-T).

For the rest of the variables, we assume continuity at these time points:
\begin{gather}
    {S}(T_{1,2}+0)= {S}(T_{1,2}-0), \nonumber \\
     {R_{C}}(T_{1,2}+0)= {R_{C}}(T_{1,2}-0), \\
     {R_{E}}(T_{1,2}+0)= {R_{E}}(T_{1,2}-0). \nonumber
\end{gather}

Note that, in clinical practice, the number of treatment applications is finite. Accordingly, in our model simulations, the number of external pulses corresponding to either TMZ or CAR-T administration is also finite. We denote by $L_1$ the number of TMZ cycles, where a standard cycle consists of five consecutive days of drug administration followed by a 23-day rest period \cite{stupp_protocl2005}. Similarly, we denote by $L_2$ the number of injections of CAR-T cells, so that the total number of CAR-T cells administered to a patient is $L_2 v$. Therefore, these external forces are applied a finite number of times rather than periodically, resulting in a five-dimensional dynamical system.

\subsection{Mathematical analysis background}\label{sec:math_back}

We now introduce a series of definitions and propositions that form the foundation for the mathematical analysis presented in Section \ref{sec:res_math_anal}.

Consider a dynamical system
\begin{equation}
\label{eq:gen_dyn_sys}
    \dot{x}=f(x),
\end{equation}
where $x\in\mathbb{R}^{n}$, $f(x)$ is a smooth vector function of $x$ that is defined in some domain $\Omega\subseteq \mathbb{R}^{n}$, and $n\in\mathbb{N}$. Since in \eqref{eqS}--\eqref{eqE} all parameters and variables are positive, its right-hand side is a smooth function in $\mathbb{R}^{5}$.

One can associate to \eqref{eq:gen_dyn_sys} the vector field $\mathcal{X}=f_{i}(x)\partial_{x_{i}}$, where $f_{i}(x)$, $i=1,\dots,n$, are components of $f$, and summation over repeated indexes is assumed.

The zero set $H(x)=0$ of a smooth function $H(x)$ defines an invariant surface of dynamical system \eqref{eq:gen_dyn_sys}, and its associated vector field $\mathcal{X}$ if there exists continuous function $k(x)$, called the cofactor, such that
\begin{equation}
\label{eq:Inv_surf_def}
    \mathcal{X}H=kH.
\end{equation}

In addition, we introduce several definitions and known results.

We denote by $\overline{\mathbb{R}}_{+}^{n}$ the nonnegative orthant of $\mathbb{R}^{n}$, i.e., the set $\overline{\mathbb{R}}_{+}^{n} = \{ x_i \geq 0, \ i=1,\dots,n \}$.

Let $\Omega$ be the open domain of definition of \eqref{eq:gen_dyn_sys}, and assume that $\overline{\mathbb{R}}_{+}^{n}\subset \Omega$.
Then, $f$ is said to be essentially nonnegative if $f_{i}(x) \geq 0$ for $i=1,\dots,n$ and $x \in \overline{\mathbb{R}}_{+}^{n}$ such that $x_{i} = 0$ (see, \cite{Bernstein1999,Haddad2005}).

The following statement guarantees the non-negativity of solutions to \eqref{eq:gen_dyn_sys} if its right-hand side is essentially non-negative (see Proposition 4.1 from \cite{Haddad2005} ):

\begin{Proposition}\label{Prop2}
    Suppose that $\overline{\mathbb{R}}_{+}^{n}\in\Omega$. Then $\overline{\mathbb{R}}_{+}^{n}$ is an invariant set of \eqref{eq:gen_dyn_sys} if and only if $f$ is essentially non-negative.
\end{Proposition}

\subsection{Virtual Patients and \textit{In Silico} Trials}\label{sec:vp_insilico_trials}

In oncology research, rapidly testing and refining therapeutic strategies is crucial for improving patient care \citep{mercieca2018importance}. However, experimental studies and clinical trials are often time-consuming and resource-intensive, which can hinder the pace of innovation. Mechanistic modeling, particularly through ordinary differential equations (ODEs), provides a more efficient alternative \citep{alfonso2020translational}. ODE-based mechanistic models are computationally efficient, enabling the simulation of numerous treatment scenarios in a fraction of the time required by other approaches, such as complex agent-based models. These models are especially useful for studying tumor dynamics and treatment evolution, as they can generate large virtual cohorts that replicate the variability seen in real-world populations \citep{dormand2018numerical}. Virtual clinical trials, or \textit{in silico} trials, have become a valuable tool in mathematical oncology \citep{gevertz2024assessing}. They help reduce the risk of clinical trial failures and inform the design of focused and high-likelihood experimental studies \citep{brown2022derisking}. By simulating various treatment regimens, researchers can systematically evaluate their effects and identify the most effective strategies for specific patient subgroups, accelerating discovery while reducing costs and addressing ethical concerns \citep{brown2022derisking}.

Section \ref{sec:model} presents our mathematical model to describe MG cancer growth and its evolution in response to TMZ and CAR-T therapy. The model parameters (Table \ref{tab:paramenters}) represent tumor characteristics, including demographic processes, response to TMZ, and response to CAR-T. By assigning specific parameter values, we can generate a virtual patient (VP)--a digital representation of a real patient. When parameter distributions are unknown, as in this study, they are typically assumed to follow random distributions \citep{ayala2021optimal}. Therefore, a VP corresponds to a specific combination of the model parameters sampled from the specified ranges in Table~\ref{tab:paramenters}.
Repeating this process $N$ times generates a virtual cohort of $N$ VPs. By simulating the responses of these VPs to specific protocols, an \textit{in silico} trial is conducted.

In Section \ref{sec:res_numerics}, we present the results of the \textit{in silico} trials in various figures using Kaplan-Meier (KM) curves and risk tables produced with MathSurv \cite{creed2020matsurv} in MATLAB\copyright\ 2023b. KM curves are a statistical tool commonly used in medical research and clinical trials to estimate and visualize the survival function from time-to-event data. These curves show the proportion of participants surviving over time and decrease at each time a death event occurs. The method also accounts for censored data, where the exact event time is unknown due to participants being lost to follow-up or the study ending before the event occurs. KM curves are used to compare survival between different groups, such as no treatment versus treatment or treatment 1 versus treatment 2, often employing statistical tests like the log-rank test. This tool enables researchers to visually (via KM curves) and numerically (via p-values of the log-rank test) compare survival differences between patient groups under different treatment strategies, providing valuable insights into optimal treatments.

Throughout the text, we use the term median virtual patient (MVP) to refer to a simulated individual whose characteristics--such as tumor parameters, treatment responses, or disease progression--represent the median point of the statistical distribution within a virtual population. Specifically, the MVP is defined by the parametrization using the median values of the distributions listed in Table \ref{tab:paramenters}. Note that for the parameter $\rho_{4}$ we used the upper bound 0.1, because the median survival time drops sharply if we increase this boundary further (see Fig. \ref{fig:f3a_d}).

We use  MATLAB\copyright\ 2023b to perform numerical calculations and visualize the results. To numerically find solutions to the Cauchy problems considered in this work, we use the ODE78 solver.

\section{Results}\label{sec:result}

Here, we present the results obtained with our mathematical model. First, Section \ref{sec:res_math_anal} presents the mathematical analysis of the system of equations \eqref{eqS}--\eqref{eqE},considering both treatments as constants. Next, in Section~\ref{sec:res_numerics}, we simulate \textit{in silico} trials using our mathematical model and present the results for the following scenarios: TMZ monotherapy (Section~\ref{sec:res_TMZ_only}), CAR-T cell monotherapy (Section~\ref{sec:res_only_CART}), combined TMZ and CAR-T therapy (Section~\ref{sec:res_combined}), and treatment responses in the presence of fast-growing TMZ-resistant cells (Section \ref{sec:res_fast_grow}).

\subsection{Mathematical Analysis}\label{sec:res_math_anal}

In this subsection, we conduct a theoretical study of our model \eqref{eqS}--\eqref{eqE}, initially considering treatment with a single dose of CAR-T and TMZ at the beginning of therapy and then considering constant doses for both treatments.

 Note that, since our model variables track the number of cancer cells, CAR-T cells and TMZ normalized concentration, by biologically relevant solutions of the system \eqref{eqS}--\eqref{eqE}, we refer to solutions--particularly equilibrium states--with all non-negative components. Moreover, we define a parameter value of the system \eqref{eqS}--\eqref{eqE} as biologically relevant if it falls within the corresponding range presented in Table \ref{tab:paramenters}.

 We begin this section by proving the existence of several invariant surfaces:

\begin{Proposition}
The system \eqref{eqS}--\eqref{eqE} has four invariant surfaces $H_{1}=S$, $H_{2}=R_{C}$, $H_{3}=R_{E}$ and $H_{4}=E$.
\label{p:p1}
\end{Proposition}
\begin{proof}
It follows directly from the definition of an invariant surface \eqref{eq:Inv_surf_def} that the functions $H_{1}=S$, $H_{2}=R_{C}$, $H_{3}=R_{E}$, and $H_{4}=E$ are invariant surfaces of {the systems} \eqref{eqS}--\eqref{eqE}, with the {respective} cofactors that are the corresponding components of the {system’s} vector field. This completes the proof.
\end{proof}

Now we proceed with the proof of non-negativeness of solutions of {the system} \eqref{eqS}--\eqref{eqE}.

Thus, using Proposition \ref{Prop2}, it is straightforward to prove the following statement:
\begin{Proposition}
If $\epsilon_{1}$ is non-negative, then the solutions of the system \eqref{eqS}--\eqref{eqE} remain non-negative for all \( t \geq 0 \), provided the initial conditions are non-negative.
\end{Proposition}
\begin{proof}
First, we demonstrate that the right-hand side of the system \eqref{eqS}--\eqref{eqE}, denoted by $\mathbf{f}(S, R_C, R_E, C, E)$, is essentially non-negative, i.e., $f_i(S, R_C, R_E, C, E) \geq 0$, when one of the variables equals zero and the others are in $\overline{\mathbb{R}}_{+}^{5}$. Specifically, we have:
\begin{align}
    & f_{1}(0, R_C, R_E, C, E) = f_{2}(S, 0, R_E, C, E) = 0, \nonumber \\
    & f_{3}(S, R_C, 0, C, E) = \epsilon_{1} E (S + R_C), \\
    & f_{4}(S, R_C, R_E, 0, E) = v, \quad f_{5}(S, R_C, R_E, C, 0) = 0. \nonumber
\end{align}
Therefore, if the conditions of this proposition are satisfied, the right-hand side of \eqref{eqS}--\eqref{eqE} is essentially non-negative. Hence, for any non-negative initial conditions, the solutions of the system \eqref{eqS}--\eqref{eqE} remain non-negative for all \( t \geq 0 \). This completes the proof.
\end{proof}

Now we consider equilibrium states of the model \eqref{eqS}--\eqref{eqE}. The following statement holds:
\begin{Proposition}
If $r_{1}\neq r_{2}$, then all isolated biologically relevant equilibrium states of \eqref{eqS}--\eqref{eqE} are unstable.
\end{Proposition}
\begin{proof}
Note that any fixed point of \eqref{eqS}--\eqref{eqE} will have the last component equal to zero, i.e., $E=0$. If, in addition $C=0$, then there is an invariant surface $K-S-R_{C}-R_{E}=0$. Its stability will be considered below, and we do not consider points that belong to it here.

The origin $O=(0,0,0,0,0)$ is an unstable degenerated node, as the corresponding eigenvalues of the Jacobi matrix are $(r_{1},r_{1},r_{2},-\rho_{1},-\lambda)$.

Suppose that $S=R_{C}=0$ and $R_{E}\neq0$. In this case, we obtain that $C_{1,2}=r_{2}(K-R_{E}^{(1,2)})/(\alpha_{2}K)$ and $R$ is a solution of a quadratic equation (recall that $K-S-R_{C}-R_{E}\neq0$). Thus, we have two equilibrium states of the form $Q_{1,2}=(0,0,R_{E}^{(1,2)},R_{C}^{(1,2)},0)$. The Jacobi matrix at $Q_{1,2}$ has a positive eigenvalue {for} all $R_{E}<K$.

Assume that $R_{C}=R_{E}=0$ and $C\neq0$. In this case, we find that $C^{(3,4)}=r_{1}(K-S^{(3,4)})/(K\alpha_{2})$ and $S^{(3,4)}$ are solutions of a quadratic equation. If we evaluate the Jacobi matrix at the equilibrium points $Q_{3,4}=(S^{(3,4)},0,0,C^{(3,4)},0)$, we find that one of the eigenvalues is positive if $S^{(3,4)}<K$.

If $S=R_{E}=0$ and $C\neq0$, we find a fixed point $Q_{5}=(0,K,0,-g_{3}-K\rho_{4}/\rho_{1},0)$. However, upon inspection, we find that \( C < 0 \), meaning that \( Q_{5} \) is not biologically relevant in the context of cancer dynamics. Furthermore, we can show that $Q_{5}$ is unstable for positive values of the parameters. This completes the proof.
\end{proof}

Now we consider the stability of the invariant hyperplane $K-S_{1}-S_{2}-R=0$ of \eqref{eqS}--\eqref{eqE} at $C=E=0$.

\begin{Proposition}
If $C=E=0$, the system \eqref{eqS}--\eqref{eqE} has an invariant surface $P_{1}=K-S-R_{C}-R_{E}$, which is locally stable.
\end{Proposition}
\begin{proof}
It is straightforward to observe that for {\eqref{eqS}--\eqref{eqE}} with $C=E=0$, any point {on} the plane
\begin{equation}
    P_{1}=K-S-R_{C}-R_{E}=0,
\end{equation}
is a fixed point.

If we compute the eigenvalues of the Jacobi matrix for \eqref{eqS}--\eqref{eqE} at $C=E=P_{1}=0$, we find that there are two zero eigenvalues and three negative eigenvalues for the parameter values of Table \ref{tab:paramenters}. Therefore, there exist stable and central manifolds near each point of $P_{1}$.

Let us demonstrate that the invariant plane $P_{1}$ is stable. Computing the eigenvalues of the Jacobi matrix for the subsystem of \eqref{eqS}--\eqref{eqE} that consists of the last two equations, at $C=E=P_{1}=0$, we obtain
\begin{align}
\lambda_{1} &= -\lambda < 0, \nonumber\\
\lambda_{2} &= -\rho_{1} - \frac{\rho_{4}K}{g_{3}} + \frac{\rho_{2}S_{1}}{g_{1} + S_{1}} + \frac{\rho_{2}S_{1}}{g_{1} + S_{1}} \leq \nonumber \\
& -\rho_{1} - \frac{\rho_{4}K}{g_{3}} + \rho_{2} + \rho_{3} < 0.
\end{align}

Now we consider the dynamics in the subspace $C=E=0$. Substituting $C=E=0$ into \eqref{eqS}--\eqref{eqE} ,we obtain a three-dimensional dynamical system
\begin{eqnarray}
\label{eq:new_model_reduction}
&S_{t}= r_{1}S_{1}\left(1-\frac{S+R_{C}+R_{E}}{K}\right), \nonumber \\
&R_{C,t}=r_{1}S_{2}\left(1-\frac{S+R_{C}+R_{E}}{K}\right), \\
&R_{E,t}=r_{2}R\left(1-\frac{S+R_{C}+R_{E}}{K}\right) \nonumber
\end{eqnarray}
which is completely integrable since it has two first integrals
\begin{equation}
    I_{1}=\frac{S}{R_{C}}, \quad I_{2}=\frac{R_{E}}{R_{C}^{r_{2}/r_{1}}}.
\end{equation}
We can reduce \eqref{eq:new_model_reduction} to
\begin{equation}
\label{eq:new_model_reduction_1}
  R_{C,t}=r_{1}R_{C}\left(1-\frac{(c_{1}+1)R_{C}+c_{2}R_{C}^{r_{2}/r_{1}}}{K}\right),
\end{equation}
where $c_{1}=S_{0}/R_{C}^{(0)}$, $c_{2}=R_{E}^{(0)}/((R_{C}^{(0)})^{r_{2}/r_{1}})$, and $S(0)=S_{0}$, $R_{C}(0)=R_{C}^{(0)}$, $R_{E}(0)=R_{E}^{(0)}$.

Consequently, the stability of a fixed point $(S_{0},R_{C}^{(0)},R_{E}^{(0)})$ satisfying $K-S_{0}-R_{C}^{(0)}-R_{E}^{(0)}=0$ is reduced to studying the stability of the fixed point $R_{C}=R_{C}^{(0)}=K-S_{0}-R_{E}^{(0)}$. Differentiating the right-hand side of \eqref{eq:new_model_reduction_1} and substituting $R_{C}=R_{C}^{(0)}$ into the result, we obtain an expression that is always negative for $r_{1}>0$, $r_{2}>0$, and $K-S_{0}-R_{C}^{(0)}-R_{E}^{(0)}=0$.  Therefore, the points that belong to the invariant plane $P_{1}$ are stable. This completes the proof.
\end{proof}

If $r_{2}=r_{1}$, there is an additional invariant line $P_{2}$ at $E=R_{C}=0$, which is defined by the following two relations:
\begin{gather}
    r_{1}\left(1-\frac{S+R_{E}}{K}\right)-\alpha_{2}C=0, \\
    -\rho_{1}+\frac{\rho_{2}S}{g_{1}+S}+\frac{\rho_{3}R_{E}}{g_{2}+R_{E}}-\frac{\rho_{4}(S+R_{E})}{g_{3}+C}=0.
\end{gather}
If we evaluate the Jacobi matrix along this line, there always exists a positive eigenvalue; hence, this invariant line is unstable.

Summarizing the above results, we conclude that in system \eqref{eqS}--\eqref{eqE} only the surface $P_{1}$, which exists at $C=E=0$, is attractive. Therefore, only with an initial dose of both CAR-T cells and TMZ, the tumor cells will proliferate and the tumor size will eventually reach the carrying capacity, i.e., the trajectories of \eqref{eqS}--\eqref{eqE} will lie on the surface $P_{1}$.

Let us demonstrate that, if we apply both treatments continuously, there is the possibility of reaching complete tumor eradication, which means that $S,\,S_{C},\,S_{R}\rightarrow0$ as $t\rightarrow\infty$. In this case, model \eqref{eqS}--\eqref{eqE} has the form
\begin{widetext}
\begin{eqnarray}
\label{eq:new_model1}
\frac{dS}{dt}&=&r_1 S \left(1-\frac{S+R_{C}+R_{E}}{K}\right)-\alpha_1 E S -\varepsilon_{1} E S-\alpha_2 CS, \label{eqS1} \\
\frac{dR_C}{dt}&=&r_1R_C\left(1-\frac{S+R_{C}+R_{E}}{K}\right)-\alpha_1E R_C-\varepsilon_{1} E R_C, \label{eqRC1}  \\
\frac{dR_E}{dt}&=&r_2R_E\left(1-\frac{S+R_{C}+R_{E}}{K}\right)-\alpha_2CR_E+\varepsilon_1 (S+R_C)E, \label{eqRE1} \\
\frac{dC}{dt}&=& \bar{v}-\rho_1C+\frac{\rho_2SC}{g_1+S} +\frac{\rho_3R_EC}{g_2+R_E}-\rho_{4}\frac{(S+R_{C}+R_{E})C}{g_{3}+C}-\alpha_3EC \label{eqC1},\\
\frac{dE}{dt}&=&\bar{E_0}-\mu E \label{eqE1},
\end{eqnarray}
\end{widetext}
where $\bar{v}$ is the dose of CAR-T cells per day ($\mbox{cells}\cdot \mbox{day}^{-1}$), and $\bar{{E_0}}$ ($\mbox{day}^{-1}$) is the normalized amount of TMZ given each day.

The following statement holds:
\begin{Proposition}
\label{p:p5}
System {\eqref{eqS1}-\eqref{eqE1}} has a stable equilibrium state of the form $(0,0,0,C^{*},E^{*})$ for the ranges of the parameters given in Table \ref{tab:paramenters} if and only if $\bar{E_0}(\alpha_{1}+\epsilon_{1})>r_{1}\mu$ and $\bar{v}>\frac{r_{2}(\bar{E_0}\alpha_{3}+\mu\rho_{1})}{\mu \alpha_{2}}$.
\end{Proposition}
\begin{proof}
{{It is straightforward to demonstrate that system \eqref{eqS1}-\eqref{eqE1} has an equilibrium state of the form $(0,0,0,C^{*},E^{*})$ if and only if
$C^{*}=\frac{\mu \bar{v}}{\alpha_{3}\bar{E_{0}}+\rho_{1}\mu}$ and $E^{*}=\frac{\bar{E_0}}{\mu}$.}}

{{Thus, we analyze the stability of the equilibrium state $Q_{6}=\left(0,0,0,\frac{\mu \bar{v}}{\alpha_{3} \bar{E_{0}}+\rho_{1}\mu},\frac{\bar{E_0}}{\mu}\right)$. The eigenvalues of the Jacobi matrix evalueted at $Q_{6}$ are given by:}}
\begin{eqnarray}
\label{eq:ct_ev}
 \lambda_{1}=&-\mu, & \lambda_{2}=-\frac{\alpha_{3}\bar{E_0}+\rho_{1}\mu}{\mu}, \nonumber \\
 \lambda_{3}=&r_{1}-\frac{(\alpha_{1}+\epsilon_{1})\bar{E_0}}{\mu}, &
 \lambda_{4}=\lambda_{3}-\frac{\alpha_{2}\bar{v}\mu}{\bar{E_0}\alpha_{3}+\mu\rho_{1}},\\
\lambda_{5}=&r_{2}-\frac{\alpha_{2}\bar{v}\mu}{\bar{E_0}\alpha_{3}+\mu\rho_{1}}. \nonumber
\end{eqnarray}
Note that all eigenvalues are real and, hence, we require that $Q_{6}$ is a stable non-degenerate node, i.e., all eigenvalues are negative.

One can see that the first two eigenvalues are always negative for biologically relevant values of the parameters. {The last eigenvalue can be considered as a constraint on the dose of CAR-T cells needed to destroy the population resistant to TMZ. The condition on the parameter $\bar{v}$ is
\begin{equation}
\label{eq:V_constr}
    \bar{v}>\bar{v_{c}}=\frac{r_{2}(\bar{E_0}\alpha_{3}+\mu\rho_{1})}{\mu \alpha_{2}},
\end{equation}
where $V_{c}$ denotes the critical dose of CAR-T cells.

If we take into account the ranges of the parameters given in Table \ref{tab:paramenters}, the maximal value of $\bar{v_{c}}$ is approximately $5.3\cdot10^{7}$ $\mbox{cells}\cdot \mbox{day}^{-1}$, which is within the biologically relevant range.}

From \eqref{eq:ct_ev}, it follows that if $\lambda_{3}<0$, then $\lambda_{4}<0$. Thus, the second condition for the stability of $Q_{6}$ is
\begin{equation}
\label{eq:a1_constr}
\bar{E_0} (\alpha_{1}+\epsilon_{1})>r_{1}\mu.
\end{equation}
This can be considered as a constraint on the efficiency of the TMZ against the sensitive population. It follows from \eqref{eq:a1_constr} that therapeutic success is favored by administering the maximum tolerable dose of chemotherapy, i.e., set $\bar{E_0}=1$. If $\bar{E_0}=1$, then the critical value of $\alpha_{1}+\epsilon_{1}$ is approximately 0.21 for the maximal value of $r_{1}$ from Table \ref{tab:paramenters}. This completes the proof.
\end{proof}

We observe that Proposition~\ref{p:p5} delineates conditions on CAR-T cell and TMZ administration that ensure stability of the tumor-free equilibrium. Consequently, by applying both therapies in the impulsive manner outlined earlier, the trajectory governed by the system in \eqref{eqS1}-\eqref{eqE1} is directed towards this equilibrium.

Let us remark that in our simulations, we employ CAR-T dosages that satisfy the condition in \eqref{eq:V_constr}. Additionally, we assume the administration of the maximum chemotherapy dosage, guaranteeing that condition \eqref{eq:a1_constr} holds over a broad range of values for $\alpha_{1}$ and $\epsilon_{1}$.

\subsection{Numerical computations} \label{sec:res_numerics}

In this subsection, we explore the applications of the model \eqref{eqS}--\eqref{eqE} in the context of \textit{in silico} trials. Specifically, we first investigate the effects of TMZ alone in Section \ref{sec:res_TMZ_only}, followed by an examination of CAR-T cell therapy as a standalone treatment in Section \ref{sec:res_only_CART}. Finally, we analyze the combined application of both therapies in Section \ref{sec:res_combined}.

We consider applications of different treatments by \textit{in silico} trials. Note that each virtual patient is characterized by a specific combination of the model parameters, with all parameters listed in Table \ref{tab:paramenters}, except for those held constant, being sampled uniformly from their respective ranges. To ensure a fair comparison across different treatment protocols, we apply treatments to the same virtual populations. Moreover, since fluctuations in the mean and median values of a distribution decay as $1/\sqrt{N}$, we fix the virtual population size to $N = 10000$ virtual patients in the following.

Throughout the remainder of the manuscript, we denote by $T_s$ the time required for a tumor to reach a critical fatal size of $10^{12}$ cells, thus representing the survival time.
Each treatment protocol is represented by an acronym that reflects the sequence and type of administered therapies. For instance, the 1C5T1C5T protocol involves one dose of CAR-T cells followed by five cycles of TMZ, another CAR-T cell dose, and five additional TMZ cycles. Similarly, the 2C protocol corresponds to the administration of two CAR-T cell injections without any chemotherapy. Survival times under different protocols are indicated by an upper index attached to $T_s$, specifying the type and number of treatments applied. For instance, $T_s^{NT}$ denotes the survival time without any treatment, $T_s^{10T}$ corresponds to ten cycles of TMZ monotherapy, and $T_s^{2C}$ represents the survival time following two CAR-T cell injections. For combination therapies, the notation reflects the full treatment sequence. Median survival times are indicated with a tilde. For instance, $\widetilde{T}_s^{5T2C5T}$ denotes the median survival time under a protocol consisting of five TMZ cycles, two CAR-T cell doses, and five additional TMZ cycles.

\subsubsection{Applications of TMZ}\label{sec:res_TMZ_only}

In this subsection, we focus exclusively on the application of TMZ as monotherapy, setting the CAR-T cell administration rate to zero ($v = 0$). The TMZ dosing schedule follows the adjuvant phase of the standard Stupp protocol \cite{stupp_protocl2005}, wherein each treatment cycle spans 28 days: TMZ is administered daily during the first five consecutive days, followed by a 23-day rest period. We assume that we administer the maximum possible dose of TMZ.
Thus, $E_0$ is set to impulsively equal 1 at the beginning of the administration days, specifically during the first $5$ days of each TMZ cycle.

\begin{figure}[ht!]
    \centering    \includegraphics[width=0.99\linewidth]{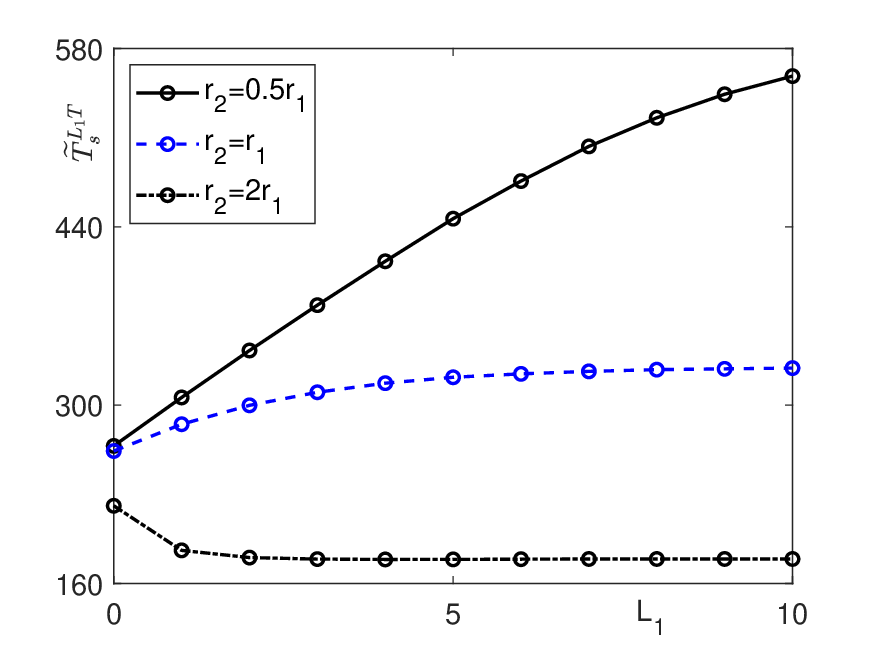}
    \caption{The dependence of the median survival time for treatment with TMZ on the number of TMZ cycles for different relationships between the sensitive ($r_{1}$) and TMZ-resistant ($r_{2}$) proliferation rates.}
    \label{fig:f6}
\end{figure}

We first investigate the dependence of the median survival time, $\widetilde{T}_{s}^{L_1T}$, on the number of TMZ cycles ($L_1$), as shown in Fig. \ref{fig:f6}. When the growth rate of TMZ-resistant cells ($r_2$) is less then or equal to that of sensitive cells ($r_1$),  increasing the number of TMZ cycles enhances survival.  Conversely, when TMZ-resistant cells grow faster than sensitive cells ($r_2 > r_1$), it is more beneficial to avoid TMZ treatment altogether. These trends are consistent with clinical evidence indicating that approximately 50\% of glioblastoma patients do not respond to TMZ treatment \cite{lee2016temozolomide}. Notably, in the case where $r_2 = 2r_1$, the median survival time remains almost constant regardless of the number of TMZ cycles. Thus, combining TMZ with CAR-T cells could be synergistically beneficial for the $r_2 = 2r_1$ case.
Motivated by these findings, we assume the maximum number of TMZ cycles ($L_1 = 10$) in subsequent analyses.

\begin{figure}[ht!]
    \centering
    \includegraphics[width=0.99\linewidth]{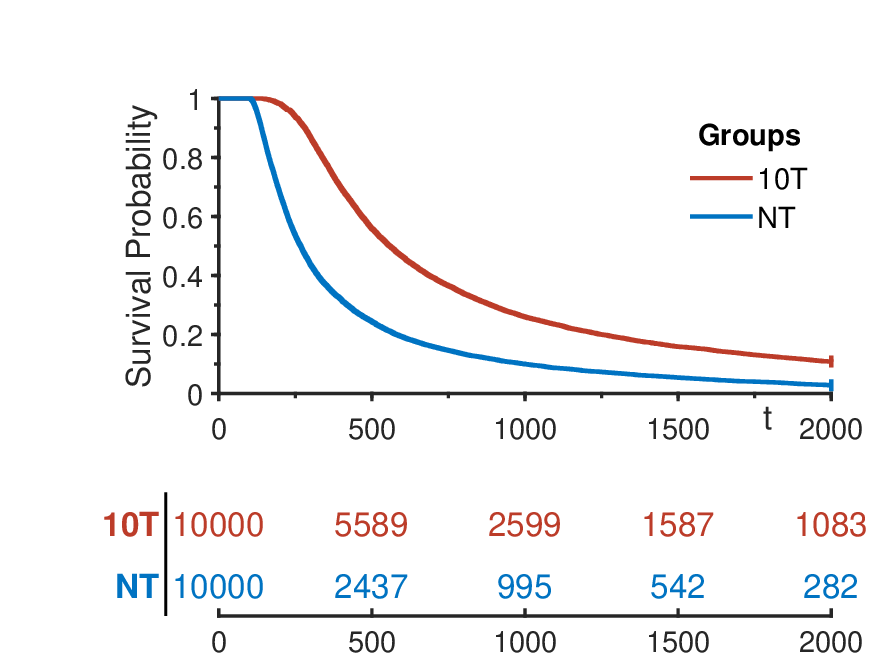}
    \caption{Kaplan–Meier curves and risk table comparing {10 TMZ cycles} (10T) and no treatment control (NT), for our \textit{in silico} trial with 10000 virtual patients.
    }
    \label{fig:f5}
\end{figure}

\begin{figure}[ht!]
    \centering
    \includegraphics[width=0.95\linewidth]{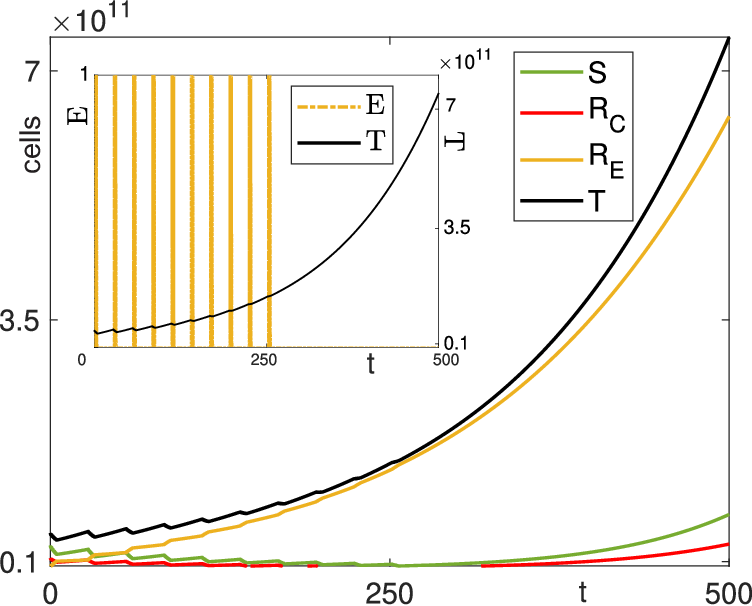}
    \caption{The dynamics of different tumor compartments and TMZ concentration for the treatment with only TMZ for the MVP.}
    \label{fig:f5a}
\end{figure}

Focusing again on the case with $r_2=0.5 r_1$,
Fig. \ref{fig:f5} shows the KM curves for the \textit{in silico} trial with 10000 virtual patients, and Fig. \ref{fig:f5a} the dynamics of the MVP. The application of 10 TMZ has a considerable impact on the sensitive part of the tumor, while the TMZ-resistant part contributes mostly to tumor growth.
The resulting median survival time under this protocol, $\widetilde{T}{s}^{10T}$, is 558 days--more than twice that observed under no treatment (NT), $\widetilde{T}{s}^{NT} = 268$ days--representing a survival benefit of 108.21\%.

\begin{table}[!h]
\caption{The Pearson correlation coefficient ($r$) and p-value between $T_{s}^{10T}$ and model parameters.}
\begin{tabular}{@{}ccc@{}}
Parameter & r & p-value \\ \hline
$r_{1}$ & -0.70 & 0.00 \\
$T_{0}$ & -0.13 & 0.00 \\
$\alpha_{1}$ & 0.11 & 0.00\\
\label{t:t3}
\end{tabular}
\end{table}

We analyze the correlation between the \textit{in silico} survival time $T_s^{10T}$ and the model parameters, as shown in Table~\ref{t:t3}, which includes only statistically significant parameters (p-value $< 0.05$) with correlation coefficients $r \geq 0.1$. This criterion is consistently applied to all subsequent Tables presenting correlation analyses throughout the manuscript.

The tumor growth rate of the sensitive population ($r_1$) emerges as the most influential factor. In addition, the initial tumor size ($T_0$) and TMZ killing efficacy ($\alpha_1$) exhibit notable correlations with survival outcomes. Other parameters display minimal influence.

\begin{figure}[!h]
    \centering
    \includegraphics[width=0.5\textwidth]{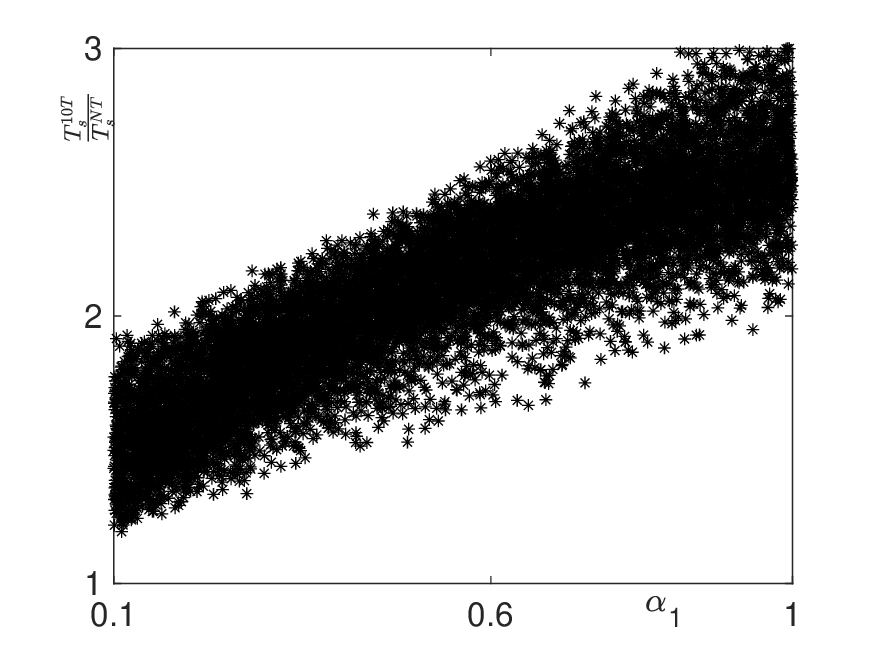}
    \caption{Correlation between the improvement of survival time with 10 TMZ cycles (10T) in comparison with no treatment (NT) $T_{s}^{10T}/T_{s}^{NT}$ and TMZ efficiency $\alpha_{1}$.}
    \label{fig:f7}
\end{figure}

We next examine correlations between the improvement in survival times with 10 TMZ cycles compared to the untreated case (defined as the ratio $T_{s}^{10T}/T_{s}^{NT}$) of the same virtual patients (same model parametrization in the 2 cases) and the model parameters.
The results, summarized in Table~\ref{t:t4}, highlight a different trend: the TMZ killing efficacy against sensitive tumor cells ($\alpha_1$) emerges as the most influential parameter. Additionally, the initial fraction of TMZ-resistant cells ($\delta_1$) shows a modest, yet non-negligible, effect.

\begin{table}[!h]
\caption{The Pearson correlation coefficient ($r$) and corresponding p-value between $T_{s}^{10T}/T_{s}^{NT}$ and model parameters.
}
\begin{tabular}{@{}ccc@{}}
Parameter & r & p-value \\ \hline
$\alpha_{1}$ & 0.88 & 0.00 \\
$T_{0}$ & 0.33 & 0.00 \\
$r_{1}$ & -0.18 & 0.00 \\
\label{t:t4}
\end{tabular}
\end{table}

The correlation between $T_{s}^{10T}/T_{s}^{NT}$ and $\alpha_{1}$ is shown in Fig.\ref{fig:f7}.


\subsubsection{Applications of CAR-T cells}\label{sec:res_only_CART}

\begin{figure}[h!]
    \centering
    \includegraphics[width=0.95\linewidth]{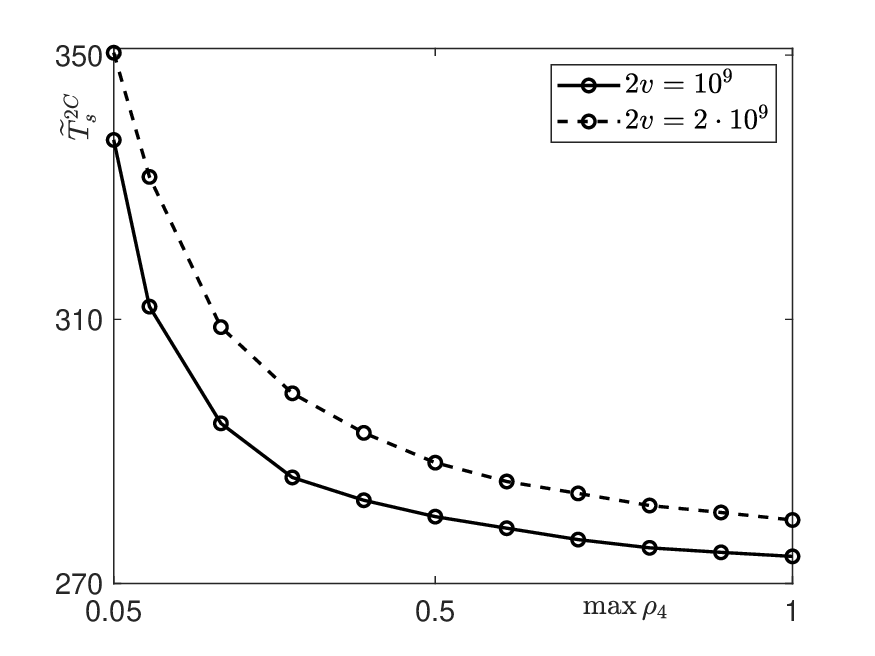}
    \caption{The dependence of the the median survival time $\widetilde{T}_{s}^{2C}$ on  $\max \rho_{4}$ for $2v=10^{9}$ (solid line) and $2v=2\cdot 10^{9}$ (dashed line).}
    \label{fig:f3a_d}
\end{figure}

In this Section, we apply only CAR-T cell monotherapy to the set of virtual patients described previously, thus $E=E_{0}=0$. Recall that one of the main challenges for effective CAR-T cell treatment is the immunosuppressive nature of the tumor microenvironment \cite{kringel2023chimeric}. For the success of CAR-T therapy, the immune properties of the tumor, such as the inactivation rate $\rho_{4}$, are important. To assess the impact of immune suppression on therapy outcomes, we analyze how variations in $\rho_{4}$ affect the median survival time. We consider the same distribution of virtual patients with different maximal values of $\rho_{4}$ from $0.05$ to $1$ (the other parameters remain the same). Note that we theoretically increase the range of $\rho_{4}$ in order to better illustrate the sharp drop in the median survival time.
Figure~\ref{fig:f3a_d} shows the dependence of the median survival time, $\widetilde{T}_{s}^{2C}$, on the upper bound of $\rho_{4}$. The results reveal a rapid drop in survival time as $\rho_{4}$ increases, approaching the untreated baseline value of 268 days. These findings suggest that CAR-T monotherapy is largely ineffective in highly immunosuppressive tumor environments. Accordingly, for all subsequent analyses in this study, we constrain the maximum value of $\rho_{4}$ to 0.1.

Now we investigate the dependence of the median survival $\widetilde{T}_{s}^{C}$ on different CAR-T cell treatment strategies, varying both the number of CAR-T cells injected ($v$) and the number of injections ($L_2$). Figure \ref{fig:f3} shows the result of the analysis.

\begin{figure}[h!]
    \centering
    \includegraphics[width=0.95\linewidth]{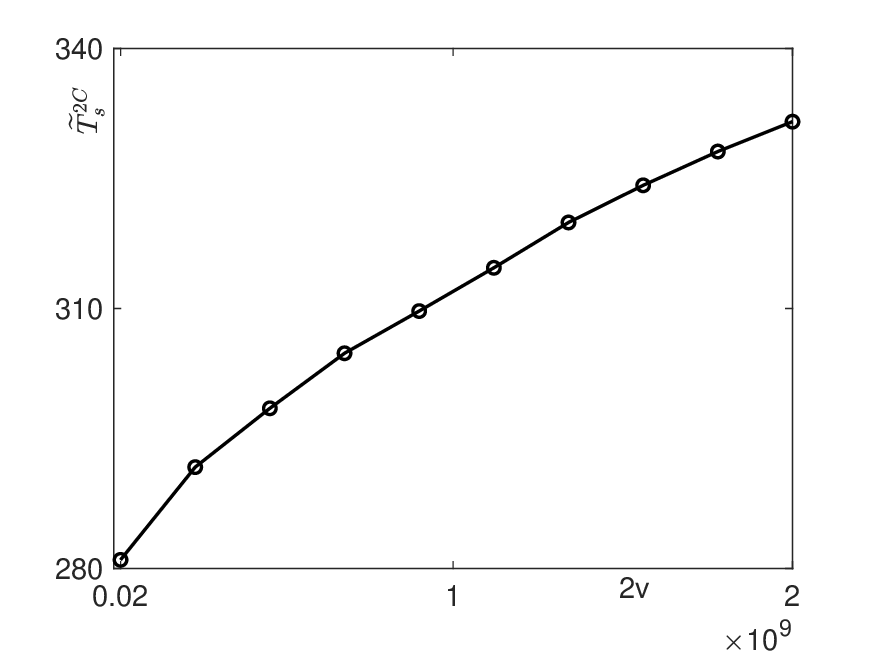}
    \includegraphics[width=0.95\linewidth]{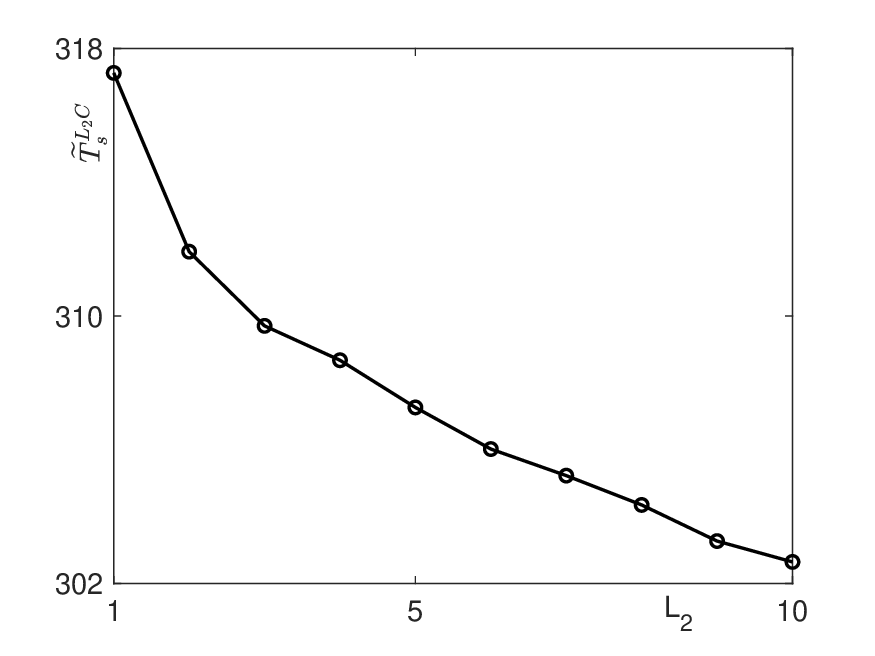}
    \caption{The dependence of $\widetilde{T}_{s}^{2C}$ on the total dosage of CAR-T cells $2v$ ($L_{2}=2$) on the amount of used CAR-T cells (upper panel). The dependence of $\widetilde{T}_{s}^{L_{2}C}$ on the number of CAR-T applications $L_{2}$ with total dose $L_{2}v=10^{9}$ (lower panel). }
    \label{fig:f3}
\end{figure}

We study the dependence of median survival times for CAR-T treatments $\widetilde{T}_{s}^{2C}$ on the total dose of CAR-T cells (see the upper panel of Fig. \ref{fig:f3}). Here, we fix the number of injections of CAR-T cells to 2 ($L_2=2$). We find that the more total CAR-T cells are given, the longer the median survival times.

 In the bottom panel of Fig. \ref{fig:f3}, we explore how distributing a fixed total dose of $10^{9}$ CAR-T cells ($v L_2 = 10^{9}$) across multiple applications (i.e., varying $L_2$) influences the median survival time.
The results show that the number of CAR-T applications has a slight impact on median survival $\widetilde{T}_{s}^{L_2C}$, the optimal approach being the administration of all CAR-T cells in a single large dose ($L_2=1$). However, given the clinical relevance of mitigating potential side effects and the relatively small difference in survival outcomes, we recommend dividing the available CAR-T cells into two applications ($L_2=2$) as a more balanced therapeutic strategy.

For patient toxicological and manufacturing constraints, we consider two different total dosages of CAR-T that we apply, $L_{2}v=10^{9}$ and $L_{2}v=2\cdot10^{9}$, and we distribute CAR-T cells in two applications, that is, $L_{2}=2$. Thus, in the first case, CAR-T cells are distributed in two equal doses of $0.5\cdot 10^9$ cells, while in the second case, CAR-T cells are distributed in two equal doses of $1\cdot 10^9$ cells.

\begin{figure}[h!]
    \centering
    \includegraphics[height=6cm]{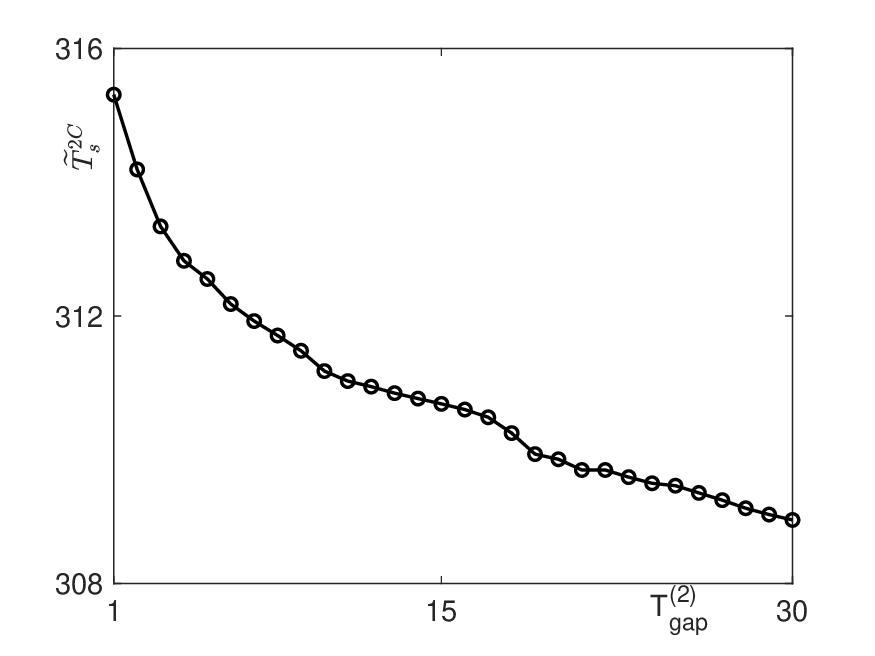}
    \caption{The dependence of the median survival time $\widetilde{T}_{s}^{2C}$ on the time gap between applications $T^{(2)}_{\text{gap}}$ for $L_{2}=2$ and $L_{2}v=10^{9}$ cells.}
    \label{fig:f3a}
\end{figure}

Fig.~\ref{fig:f3a} illustrates the dependence of the median survival time on the interval between two CAR-T cell administrations. The results show a slight decrease in survival outcomes as the time gap between injections increases. Despite this modest negative trend, introducing a delay between CAR-T doses can be clinically advantageous for reducing potential treatment-related toxicity.  Considering that the typical lifespan of CAR-T cells in the human body is approximately 7 days, we distribute CAR-T cells in two injections with a 7-day interval between them.

\begin{figure}[ht!]
    \centering
    \includegraphics[height=6cm]{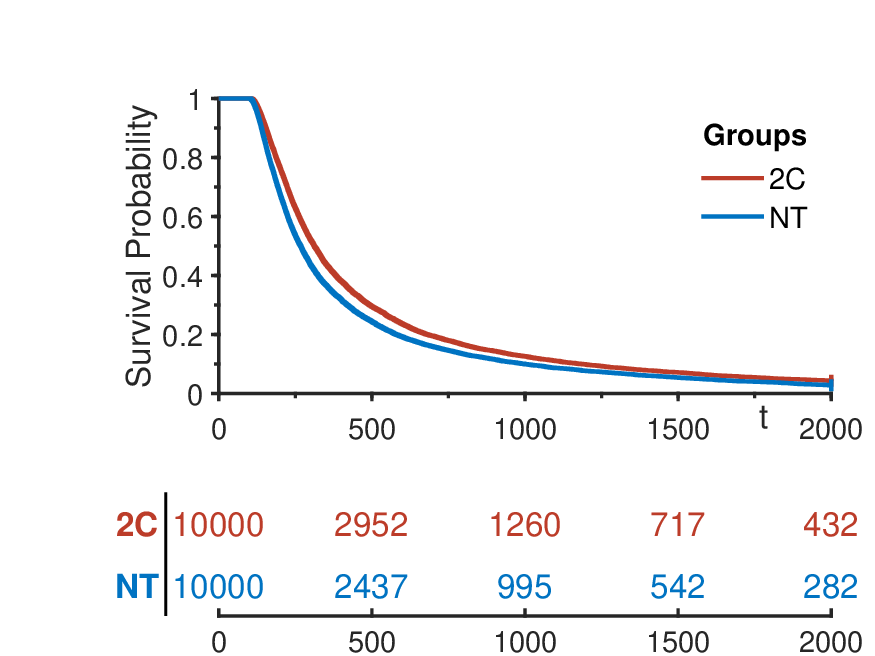}
    \includegraphics[height=6cm]{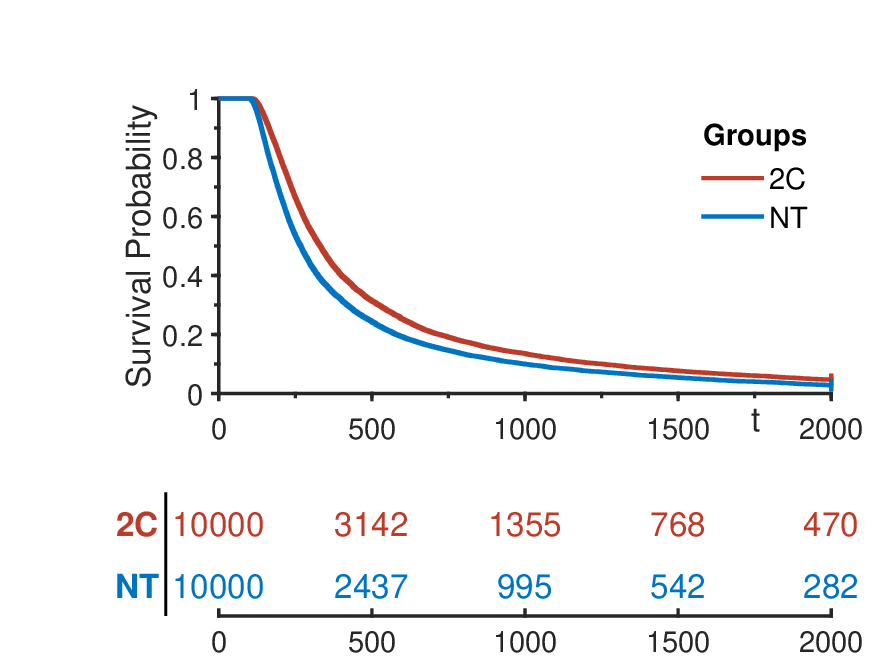}
    \caption{KM curves and risk tables comparing CAR-T cell treatment (2C) with no treatment control (NT) for $L_{2}v = 10^{9}$ (top panel) and $L_{2}v = 2 \cdot 10^{9}$ (bottom panel).
    }
    \label{fig:f1}
\end{figure}

\begin{figure}[ht!]
    \centering
    \includegraphics[width=0.95\linewidth]{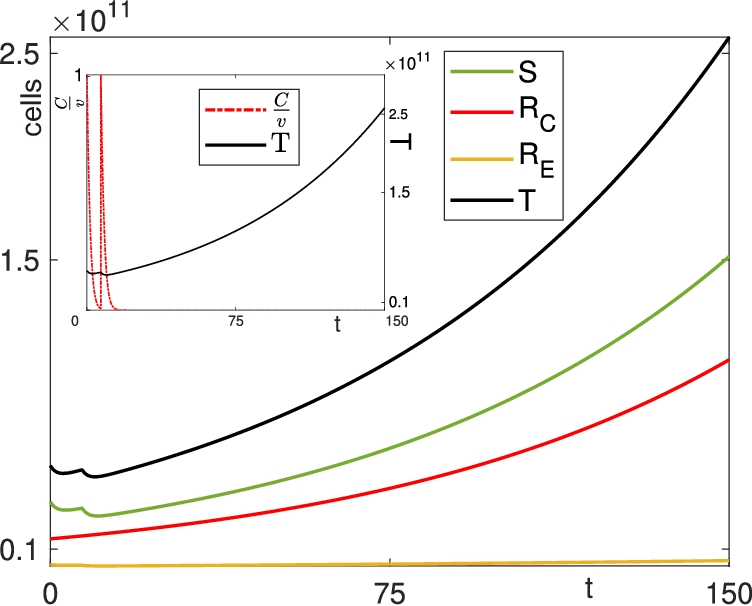}
    \caption{The dynamics of different tumor compartments and CAR-T cells for the treatment with only CAR-T cells (2C protocol) for the MVP and $L_{2}v=10^{9}$.}
    \label{fig:f2}
\end{figure}

By performing an \textit{in silico} trial applying only CAR-T cell therapy to the virtual cohort ($N=10000$ virtual patients), there is a slight positive effect on survivals (see KM curves in Fig. \ref{fig:f1}). The median survival time for the untreated population (NT) is 268 days. With a total dose of $10^9$ CAR-T cells, the median survival increases to 312 days, corresponding to a 16.42\% improvement. Doubling the CAR-T dose to $2 \cdot 10^9$ cells yields a median survival of 332 days, reflecting a 23.88\% gain relative to no treatment.

In Fig. \ref{fig:f2}, we present the dynamics of the tumor components and CAR-T cells for the MVP. The result shows that CAR-T cells do not proliferate and exert only a minimal impact on tumor dynamics for the MVP, characterized by median parameter values.

\begin{figure}[h!]
    \centering
    \includegraphics[width=0.95\linewidth]{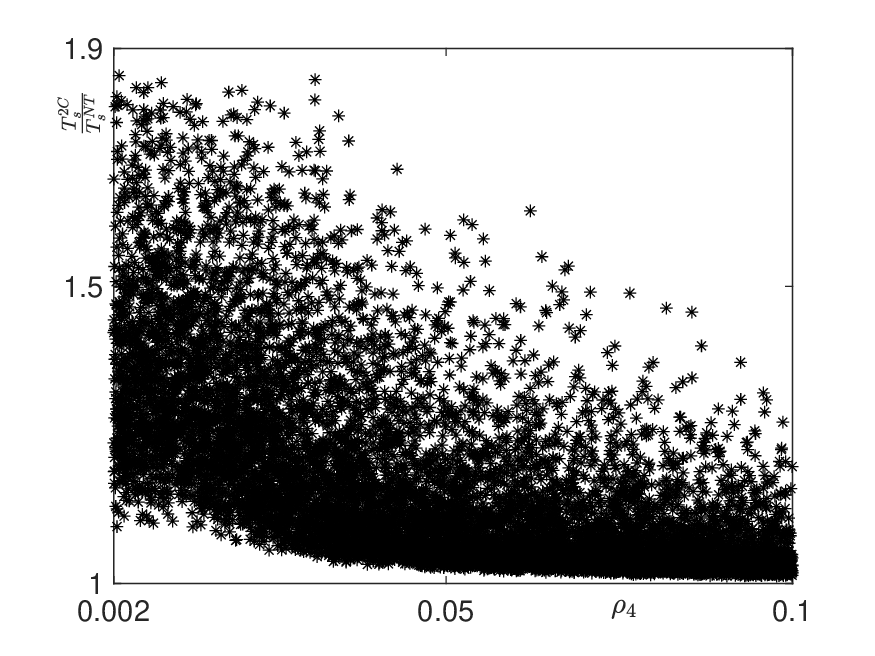}
    \caption{Correlation between ${T_{s}^{2C}}/{T_{s}^{NT}}$ and $\rho_4$ for $2v=10^{9}$.}
    \label{fig:f3a_c}
\end{figure}

\begin{table}[!h]
\caption{The statistically significant correlations ($r$ and p-value) between ${T}_{s}^{2C}$ and the model parameters { for $2v=10^{9}$}.}
\begin{tabular}{@{}ccc@{}}
Parameter & r & p-value \\ \hline
$r_{1}$ & -0.68 & 0.00 \\
$T_{0}$ & -0.21 & 0.00 \\
\label{t:t1}
\end{tabular}
\end{table}

Table \ref{t:t1} shows that survival time is strongly correlated with the growth rate $r_{1}$ of the sensitive population, and weakly correlated with the initial tumor size $T_{0}$ and the tumor inactivation rate $\rho_{4}$. If instead of $T_{s}^{2C}$ we consider the correlation between the fraction of $T_{s}^{2C}$ for CAR-T cell monotherapy over $T_{s}^{NT}$ without treatment, we find that the four most important parameters for survival improvements are the tumor inactivation rate $\rho_{4}$, the fraction of cells resistant to CAR-T $\delta_{2}$, and the mitotic stimulation of CAR-T cells by sensitive ($\rho_{2}$) and resistant ($\rho_{3}$) tumor cells (see Table \ref{t:t2}). However, also the mean life of activated CAR-T cells ($\rho_{1}$) and the sensitive cell growth rate ($r_{1}$) are statistically correlated, but with limited impact.

\begin{table}[!h]
\caption{The statistically significant correlations ($r$ and p-value) between  ${T}_{s}^{2C}/T_{s}^{NT}$ and the model parameters
 { for $2v=10^{9}$}}
\begin{tabular}{@{}ccc@{}}
Parameter & r & p-value \\ \hline
$\rho_{4}$ & -0.66 & 0.00 \\
$\delta_{2}$ & -0.38 & 0.00 \\
$\rho_{2}$ & 0.2 & 0.00 \\
$\rho_{3}$ & 0.2 & 0.00 \\
$T_{0}$ & -0.15 & 0.00 \\
\label{t:t2}
\end{tabular}
\end{table}

This relationship is further illustrated in Fig.~\ref{fig:f3a_c}, where a strong negative correlation between $\rho_4$ and ${T_{s}^{2C}}/{T_{s}^{NT}}$ is evident. Notably, lower values of $\rho_4$ are associated with greater therapeutic efficacy of CAR-T cell monotherapy relative to no treatment.

\subsubsection{Combined treatments: TMZ and CAR-T cells}\label{sec:res_combined}

In this section, we explore several combined treatments. The goal is to understand how TMZ and CAR-T cells should be optimally combined. Specifically, we evaluate a set of naive combination protocols in which two CAR-T cell injections ($L_{2} = 2$) are administered alongside ten TMZ cycles ($L_{1} = 10$). The timing of CAR-T administration is varied across three scenarios: before, between, and after the TMZ cycles. We consider two values for the total number of CAR-T cells injected: $2 v=10^{9}$ and $2v=2\cdot 10^{9}$. As a control treatment, we use the treatment with 10 TMZ cycles, since {{TMZ is typically administered to patients with MG}}. The median survival time under this 10T protocol is 558 days for both CAR-T cell dose distributions. The objective is to assess whether combined treatments can yield improvements over this baseline.

\begin{table}[!h]
\centering
\caption{Median survival times for different combined protocols with 10 ($L_1$) TMZ cycles and 2 ($L_2$) CAR-T cell injections with $r_2=0.5r_1$, for $L_2v=10^9$ and $L_2v=2\cdot 10^9$.}\label{table_ct}
\begin{tabular}{ccc}
   Protocol &$\widetilde{T}_{s}$, days, ${2}v=10^{9}$ & $\widetilde{T}_{s}$, days, ${2}v=2\cdot10^{9}$  \\ \hline
NT & 268 & 268 \\
2C & 312 & 332 \\
10T & 558 & 558 \\
5T2C5T & 652 & 689 \\
2C10T & 641 & 665 \\
1C5T1C5T & 653 & 688 \\
5T1C5T1C & 638 & 673 \\
10T2C & 597 & 636 \\
1C10T1C & 641 & 670
\end{tabular}
\end{table}

\begin{figure}[!h]
    \centering
    \includegraphics[width=0.9\linewidth]{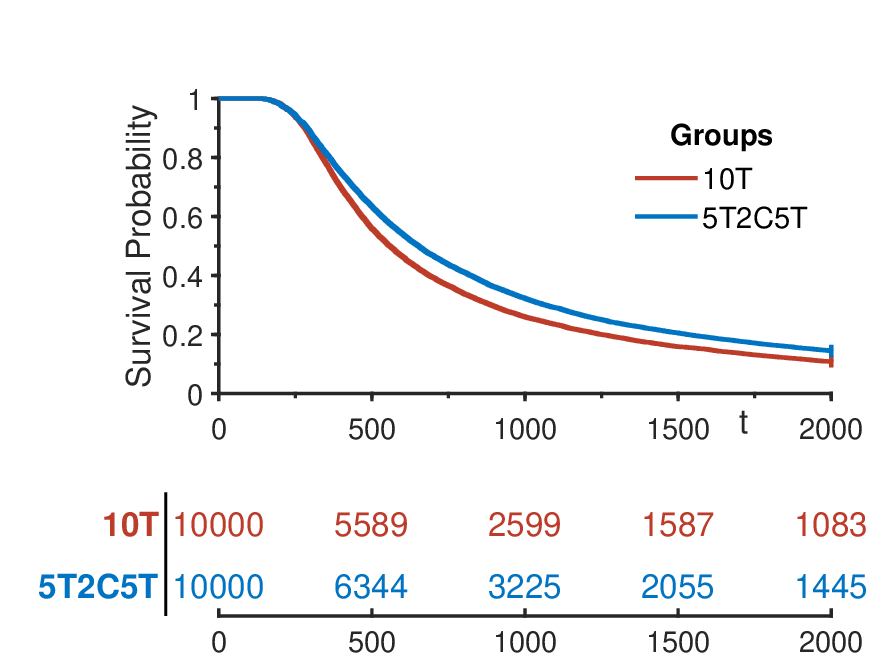}
       \includegraphics[width=0.9\linewidth]{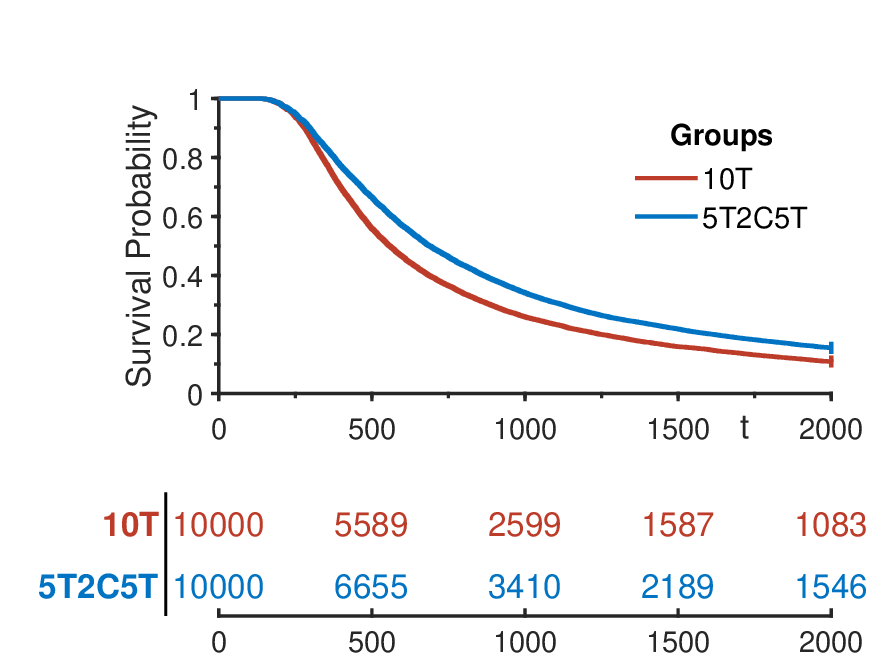}
    \caption{KM curves and risk tables comparing the 5T2C5T and 10T protocols, for $2v=10^{9}$ (top) and $2v=2\cdot 10^{9}$ (bottom). }
    \label{fig:f8}
\end{figure}

\begin{figure}[!h]
    \centering
    \includegraphics[width=0.9\linewidth]{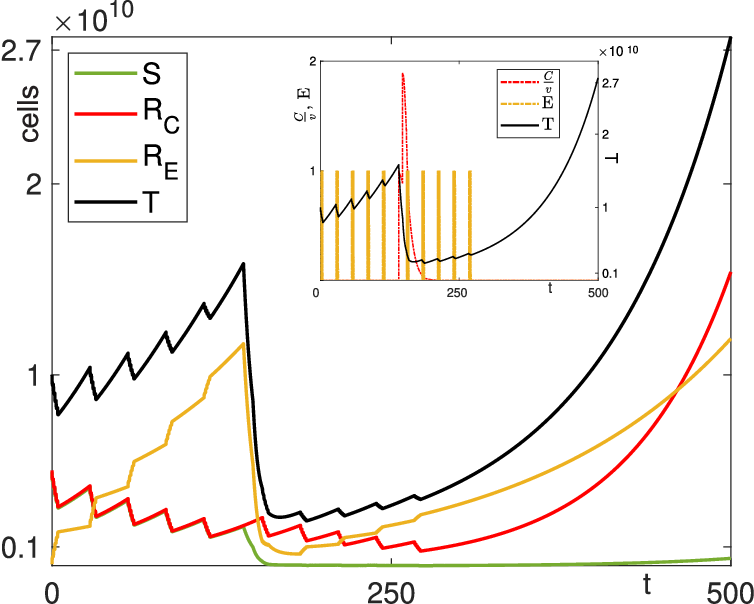}
    \caption{Dynamics for the MVP under the 5T2C5T protocol for $2v=10^{9}$. }
    \label{fig:f9}
\end{figure}

A summary of the median survival outcomes for each combined protocol is presented in Table \ref{table_ct}.

{We begin by analyzing the 5T2C5T protocol, in which two CAR-T cell injections are administered between two sets of five TMZ cycles. This protocol yields a median survival time of 652 days for a total CAR-T dose of $L_2 v = 10^9$, and 689 days for $L_2 v = 2 \cdot 10^9$. These correspond to substantial improvements of 94 days (16.85\%) and 131 days (23.48\%) relative to the 10T control protocol (558 days). The corresponding KM curves are shown in Fig.~\ref{fig:f8}.}

\begin{table}[!h]
\caption{The statistically significant correlations ($r$ and p-value) between ${T}_{s}^{5T2C5T}$ and model parameters for $2v=10^{9}$.}
\begin{tabular}{@{}ccc@{}}
Parameter & r & p-value \\ \hline
$r_{1}$ & -0.69 & 0.00 \\
$\alpha_{1}$ & 0.13 & 0.00 \\
$T_{0}$ & -0.12 & 0.00 \\
\label{t:t5a}
\end{tabular}
\end{table}

Figure \ref{fig:f9} shows the dynamics of {the model variables} for the MVP under the 5T2C5T protocol for $2v=10^{9}$. Notably, the CAR-T cells effectively target the TMZ-resistant tumor population, resulting in improved tumor control compared to TMZ monotherapy.

\begin{table}[!h]
\caption{The statistically significant correlations ($r$ and p-value) between  ${T}_{s}^{5T2C5T}/T_{s}^{NT}$ and model parameters {{for $2v=10^{9}$.}}}
\begin{tabular}{@{}ccc@{}}
Parameter & r & p-value \\ \hline
$\alpha_{1}$ & 0.69 & 0.00 \\
$\rho_{4}$ & -0.36 & 0.00 \\
$r_{1}$ & -0.29 & 0.00 \\
$T_{0}$ & 0.18 & 0.00 \\
$\epsilon_{1}$ & 0.17 & 0.00\\
$\rho_{2}, \rho_{3}$ & 0.12 & 0.00 \\
$\delta_{2}$ & -0.11 & 0.00 \\
\label{t:t6a}
\end{tabular}
\end{table}

In Tables \ref{t:t5a}--\ref{t:t6a}, we present correlations between survival time and its improvement in comparison to the absence of treatment for the 5T2C5T protocol. The two Tables show that the most important parameters that affect the protocol performance are $r_{1}$, $\alpha_{1}$, $T_{0}$ and $\rho_{4}$. Among these, the TMZ killing efficiency $\alpha_{1}$ emerges as the most critical determinant of treatment efficacy under the 5T2C5T combination strategy.

\begin{figure}[!h]
    \centering
    \includegraphics[width=0.9\linewidth]{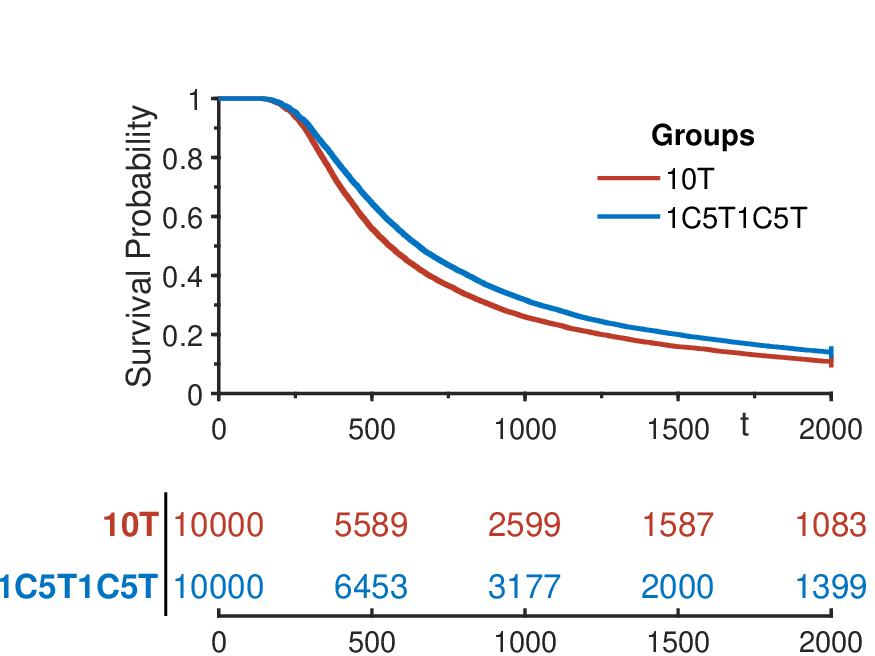}
       \includegraphics[width=0.9\linewidth]{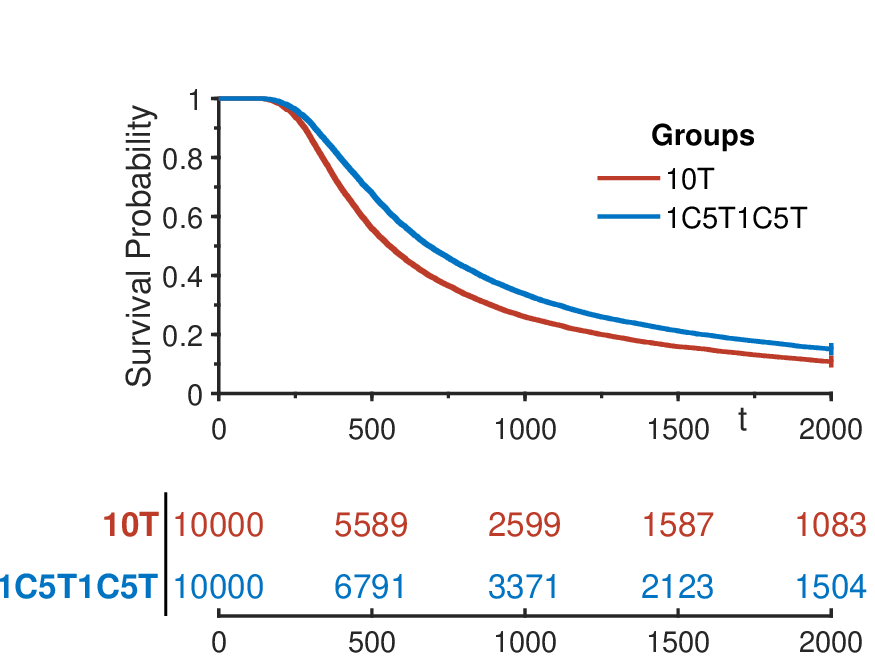}
    \caption{KM curves and risk tables comparing the 1C5T1C5T and 10T protocols for $L_2v=10^{9}$ (top) and $L_2v=2\cdot 10^{9}$ (bottom).}
    \label{fig:f12}
\end{figure}

\begin{figure}[!h]
    \centering
    \includegraphics[width=0.9\linewidth]{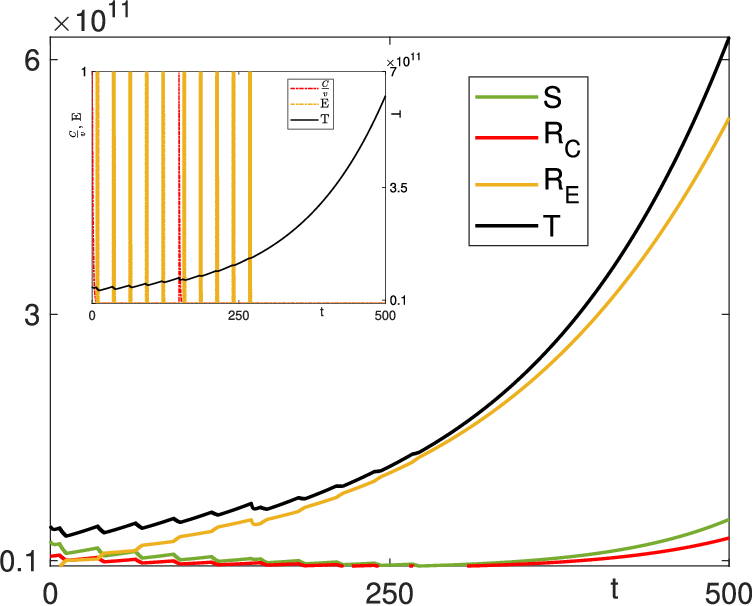}
    \caption{Dynamics for the MVP under the 1C5T1C5T protocol for $2v=10^{9}$.  }
    \label{fig:f13}
\end{figure}

Next, we consider the 1C5T1C5T protocol. The resulting median survival times are 653 days for $L_2v=10^{9}$, and 688 days for $L_2v=2\cdot 10^{9}$ (Fig. \ref{fig:f12}).  These outcomes correspond to gains of 95 days (17.03\%) and 130 days (23.30\%) relative to the 10T protocol. We also show the dynamics of tumor cells for this protocol applied to the MVP in Fig. \ref{fig:f13}. While the 1C5T1C5T strategy achieves survival benefits comparable to those of the 5T2C5T protocol, it appears less effective in targeting TMZ-resistant tumor cells. This observation is supported by the correlation results shown in Tables~\ref{t:t5b}--\ref{t:t6b}, which highlight the relative sensitivity of this protocol to resistance-related parameters.

\begin{figure}[]
    \centering
    \includegraphics[width=0.9\linewidth]{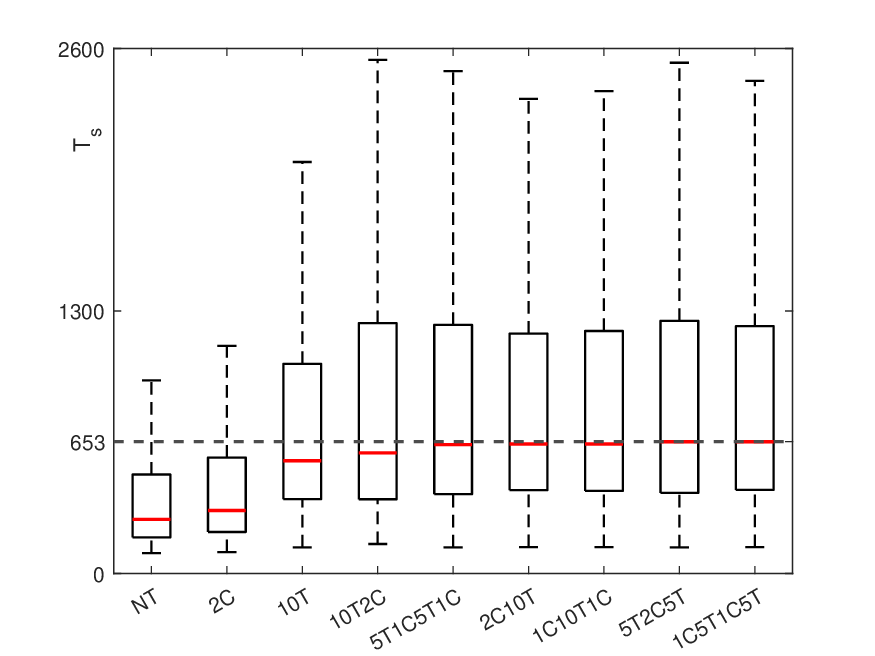}
       \includegraphics[width=0.9\linewidth]{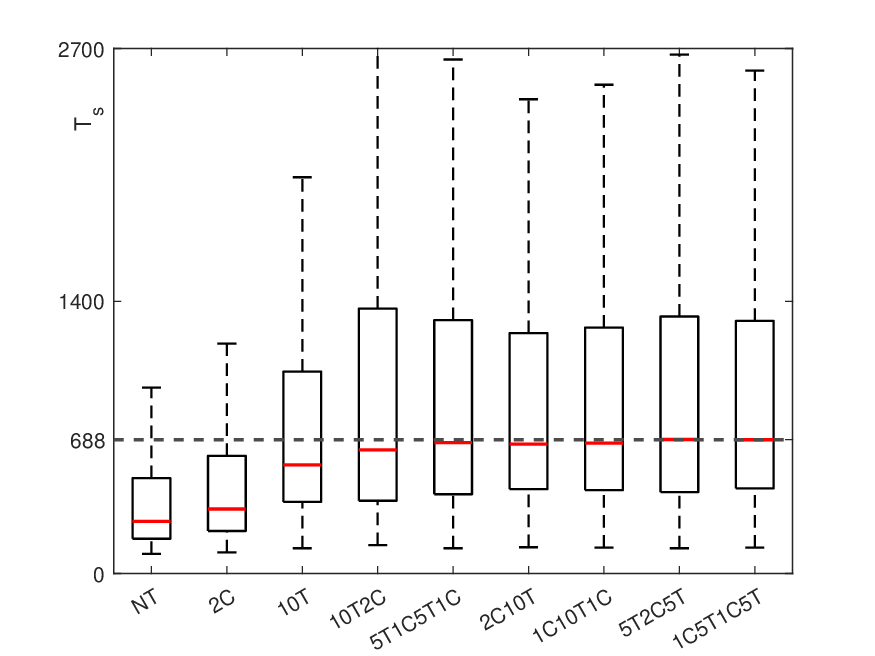}
    \caption{Box plot showing the efficacy of combined protocols for $2v=10^{9}$ (top) and $2v=2\cdot10^{9}$ (bottom).}
    \label{fig:f18}
\end{figure}

\begin{table}[!h]
\caption{The statistically significant correlations ($r$ and p-value) between ${T}_{s}^{1C5T1C5T}$ and model parameters for $2v=10^{9}$.}
\begin{tabular}{@{}ccc@{}}
Parameter & r & p-value \\ \hline
$r_{1}$ & -0.69 & 0.00 \\
$T_{0}$ & -0.13 & 0.00 \\
$\alpha_{1}$ & 0.12 & 0.00 \\
\label{t:t5b}
\end{tabular}
\end{table}

\begin{table}[!h]
\caption{The statistically significant correlations ($r$ and p-value) between  ${T}_{s}^{1C5T1C5T}/T_{s}^{NT}$ and model parameters {for $2v=10^{9}$}.}
\begin{tabular}{@{}ccc@{}}
Parameter & r & p-value \\ \hline
$\alpha_{1}$ & 0.66 & 0.00 \\
$\rho_{4}$ & 0.42 & 0.00 \\
$r_{1}$ & -0.2 & 0.00 \\
$\epsilon_{1}$ & 0.19 & 0.00\\
$T_{0}$ & 0.19 & 0.00 \\
$\rho_{2}, \rho_{3}$ & 0.14 & 0.00 \\
$\delta_{{2}}$ & -0.14 & 0.00 \\
\label{t:t6b}
\end{tabular}
\end{table}

Figure \ref{fig:f18} presents box plots of the survival times for all combined treatment protocols, using CAR-T cell doses of $L_2 v = 10^9$ and $L_2 v = 2 \cdot 10^9$. These plots illustrate the effectiveness of each protocol, highlighting that the 5T2C5T and 1C5T1C5T protocols are the most successful in terms of survival outcomes.

\textbf{Protocol comparison}. Here, we compare the combined protocols described earlier. We begin by calculating the pairwise correlations between the survival times of all protocols and find that the correlation coefficients are consistently around 0.99, with p-values close to zero. This high degree of correlation suggests that the protocols yield similar survival outcomes. This finding is further supported by the following analysis. For each patient, we assign a set of protocols. This set is constructed as follows: first, the protocol yielding the best survival outcome is assigned to the patient. Then, we include any protocol with a survival time that is no more than 5\% or 10\% shorter than the best protocol. This approach allows us to identify which virtual patients can benefit from multiple protocols and which might require a more individualized approach. When using a 5\% margin for protocol equivalence, we find that 6098 patients--approximately 61\% of the population--can be treated with any of the protocols. With a 10\% margin, 7528 patients--approximately 75\% of the population--are eligible for treatment with any protocol. These results indicate that more than half of the patients can be treated with any of the protocols with comparable outcomes.

\begin{table}[!h]
\caption{Shift in the median values of the model parameters for virtual patients with 3 suitable protocols.}
\begin{tabular}{@{}cccc@{}}
Parameter  & Shift in median value, \%  \\ \hline
$\rho_{4}$ & -32  \\
$r_{1}$ & 28  \\
$\rho_{2}, \rho_{3}$ & 17  \\
$\delta_{2}$ & -14  \\
$\alpha_{1}$ & -8  \\
$\epsilon_{1}$ & 7 \\
$T_{0}$ & 7  \\
 $\rho_{1}$ & -1 \\
  $\delta_{1}$ & 1 \\
\label{t:t_shift}
\end{tabular}
\end{table}

On the other hand, it is evident that some patients experience substantial improvements in survival time depending on the choice of protocol. To identify the regions of the parameter space associated with these patients, we select those for whom 1, 2, or 3 protocols yield similar results within a 10\% margin. This group includes approximately 1000 patients (approximately 10\% of the population). The shifts in the median values of key parameters for these patients are summarized in Table~\ref{t:t_shift}.  Our analysis reveals that patients with lower immune suppression, lower immune resistance, and higher growth rates or mitotic stimulation may benefit from a more individualized treatment approach, as these factors seem to play a significant role in determining which protocol is most effective for them.

\begin{figure}[]
    \centering
    \includegraphics[width=0.9\linewidth]{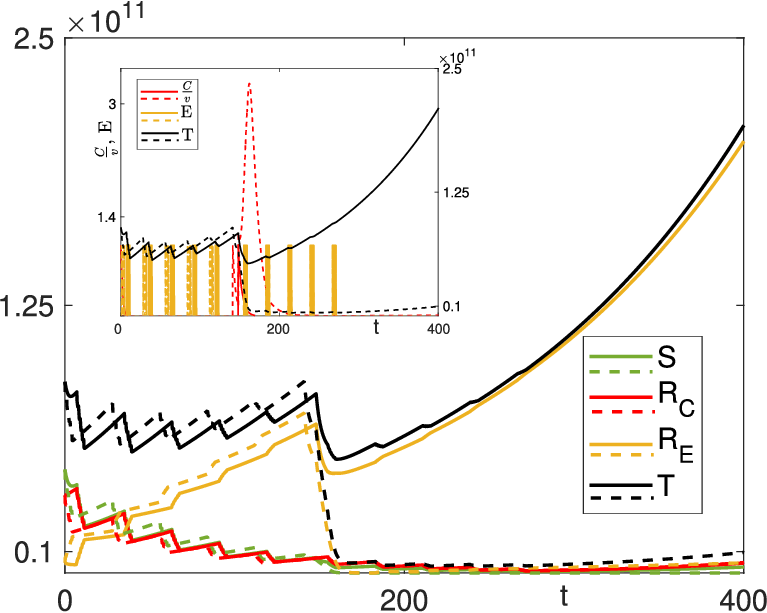}
       \includegraphics[width=0.9\linewidth]{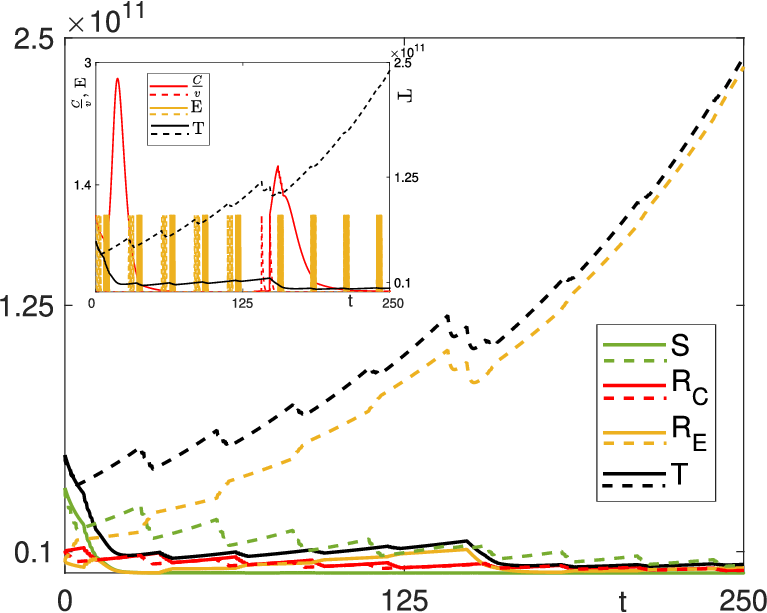}
    \caption{A comparison of tumor dynamics for 2 different virtual patients under the 5T2C5T (dashed line) and the 1C5T1C5T protocols (solid line) shows that 5T2C5T outperforms 1C5T1C5T (upper panel) and vice versa (bottom panel).   }
    \label{fig:f18b}
\end{figure}

Differences in protocol performance can, in some cases, be attributed to favorable conditions for CAR-T cell proliferation. For instance, although two of the best-performing protocols at the population level, 5T2C5T and 1C5T1C5T, show overall effectiveness, there are still individual patients for whom one protocol outperforms the other. Figure \ref{fig:f18b} illustrates tumor and treatment dynamics for two representative patients, where one of these protocols significantly outperforms the other. Notably, CAR-T cell proliferation is observed during the better-performing protocol.

 \begin{figure}[]
     \centering
     \includegraphics[width=0.9\linewidth]{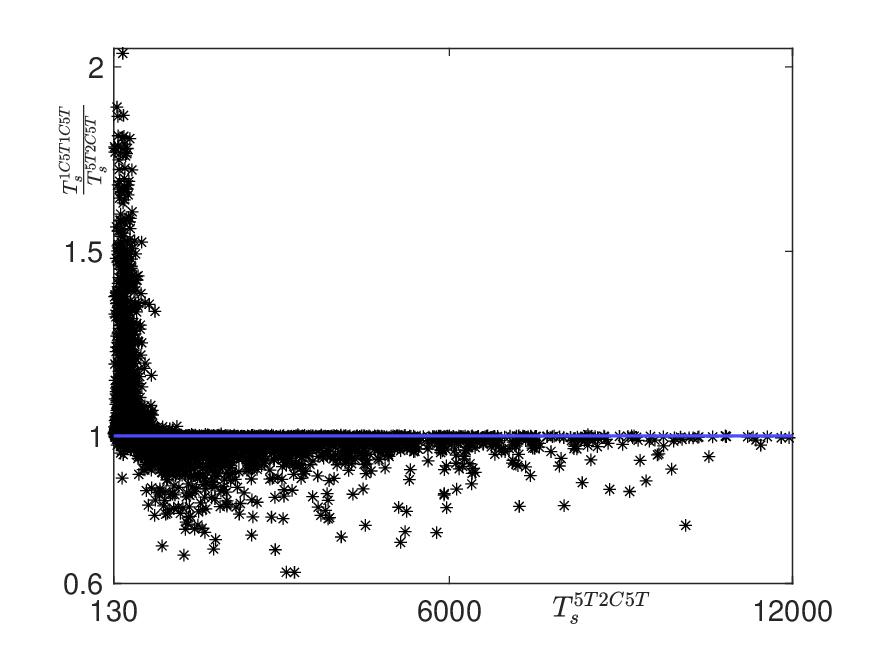}
     \caption{The correlation between the fraction $\frac{T_{s}^{1C5T1C5T}}{T_{s}^{5T1C5T}}$ and $T_{s}^{5T2C5T}$. }
     \label{fig:f18a}
 \end{figure}

Here, we focus on comparing the two best-performing protocols: 5T2C5T and 1C5T1C5T. Figure \ref{fig:f18a} shows the correlation between the difference $T_{s}^{1C5T1C5T}-T_{s}^{5T1C5T}$ and $T_{s}^{5T2C5T}$. Positive values of this difference indicate better performance of the 1C5T1C5T protocol, while negative values correspond to better performance of the 5T2C5T protocol.

From Fig. \ref{fig:f18a}, it is evident that the 1C5T1C5T protocol performs better for virtual patients with lower survival times, whereas the 5T2C5T protocol is more advantageous for those with longer survival times.

When comparing the performance of the 5T2C5T and 1C5T1C5T protocols in greater detail, we find that survival times for both protocols are almost linearly correlated, with a correlation coefficient of 0.99 and a p-value of 0. This suggests that, on average, if a virtual patient responds well to one protocol, that patient is likely to respond similarly to the other.

However, as shown in Fig. \ref{fig:f18a}, there are some differences in survival times for certain virtual patients. In particular, the group of patients where one protocol outperforms the other by more than 10 days consists of 3664 patients, approximately one-third of the virtual cohort. When the threshold for overperformance increases to 30 days, the number of patients decreases to 1840, roughly 18\% of the virtual cohort. Therefore, only a small fraction of virtual patients may derive substantially greater benefit from one protocol over the other.

 \begin{table}[!h]
 \caption{The statistically significant correlations ($r$ and p-value) between $T_{s}^{1C5T1C5T}/T_{s}^{5T2C5T}$ and model parameters for $2v=10^{9}$.}
 \begin{tabular}{@{}ccc@{}}
 Parameter & r & p-value \\ \hline
 $r_{1}$ & 0.32 & 0.00 \\
 $\alpha_{1}$ & -0.18 & 0.00 \\
 $\rho_{4}$ & -0.17 & 0.00 \\
 $\delta_{{{2}}}$ & -0.1 & 0.00 \\
 \label{t:ttt}
 \end{tabular}
 \end{table}

To understand the reasons behind the differences in performance, we investigate the correlations between the ratio $T_{s}^{1C5T1C5T}/T_{s}^{5T2C5T}$ and the model parameters. The results are presented in Table \ref{t:ttt}. Consistent with previous findings, the growth rate ($r_{1}$) emerges as the most important parameter. This finding reinforces that overall survival is largely determined by the intrinsic proliferative capacity of the tumor. Furthermore, the 1C5T1C5T protocol tends to perform better for virtual patients with shorter survival times, while the 5T2C5T protocol is more advantageous for those with longer survival times.

 \begin{figure}[]
     \centering
     \includegraphics[width=0.9\linewidth]{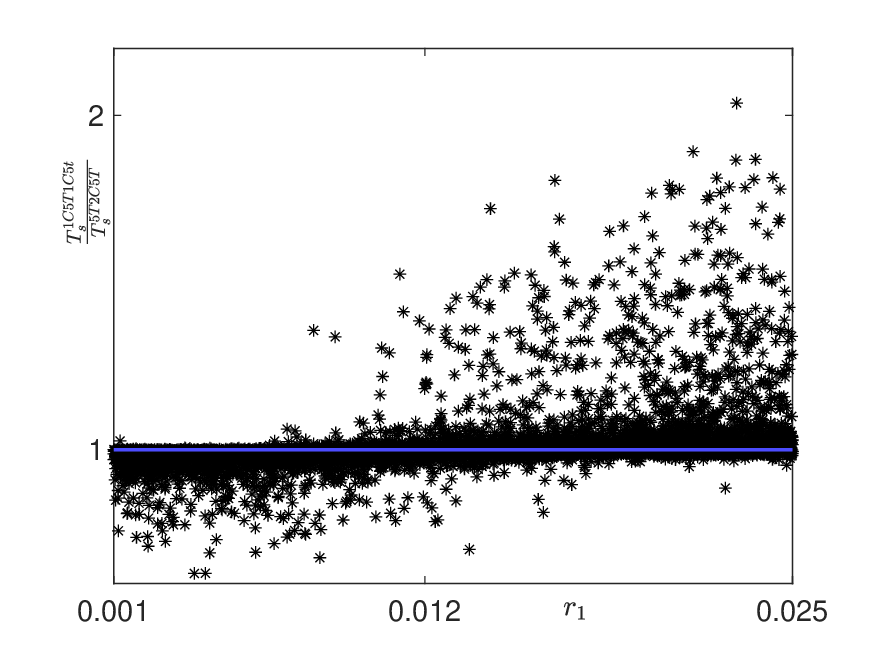}
        \includegraphics[width=0.9\linewidth]{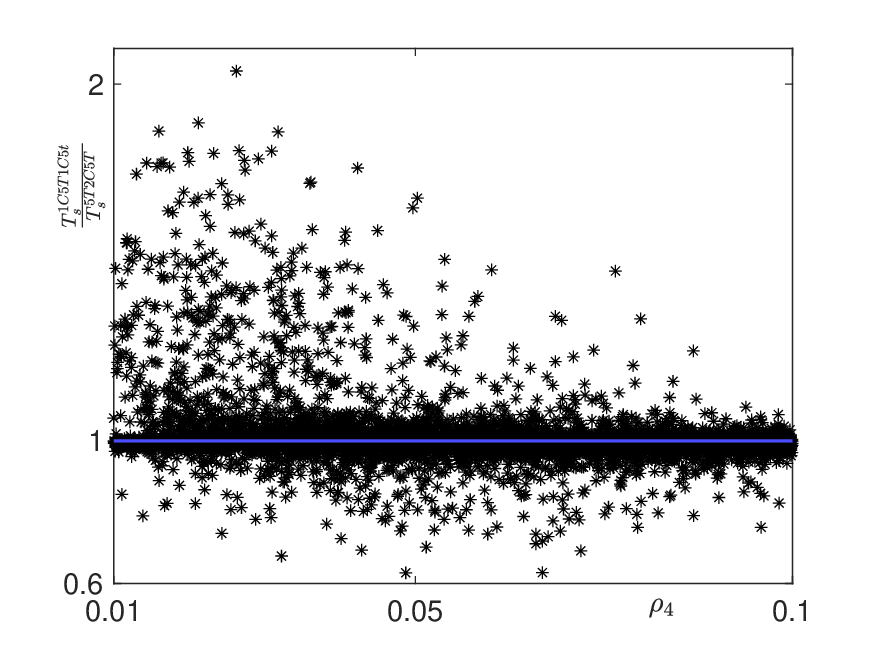}
     \caption{The correlation between $\dfrac{T_{s}^{1C5T1C5T}}{T_{s}^{5T2C5T}}$ and the parameter $r_1$ (top panel) or the parameter $\rho_4$ (bottom panel). }
     \label{fig:ttt}
 \end{figure}

Figure \ref{fig:ttt} presents the correlations between the ratio $T_{s}^{1C5T1C5T}/T_{s}^{5T2C5T}$ and the parameters $r_{1}$ and $\rho_{4}$. We observe that the 5T2C5T protocol performs better for tumors with lower $r_{1}$, and there is no correlation between its performance and $\rho_{4}$. In contrast, the 1C5T1C5T protocol is more effective for tumors with higher $r_{1}$ and lower $\rho_{4}$.
These findings highlight the potential for personalized treatment strategies, emphasizing the importance of tailoring therapies based on specific tumor characteristics, which justifies further research into personalized treatment protocols.

\subsubsection{Fast-growing TMZ-resistant cells}\label{sec:res_fast_grow}

In the previous Sections, it was assumed that $r_{2}=r_{1}/2$, meaning that TMZ-resistant cells proliferate at a slower rate than the other two populations \cite{campos2014aberrant,stepanenko2016temozolomide,yuan2018abt,dai2018scd1}.  However, in some cases, it has been observed that TMZ-resistant cells ($R_{E}$) grow similarly or faster than sensitive cells \cite{gupta2014discordant,stepanenko2016temozolomide,dai2018scd1,Delobel_PLOS}.

\begin{table}[]
\begin{center}
\caption{Median survival times $\widetilde{T}_{s}^{v}$ in days for different protocols with $r_2=r_1$ (second and third columns) and $r_2=2r_1$ (fourth and fifth columns), for $L_2v=10^9$ and $L_2v=2\cdot 10^9$. }
\vspace{0.25cm}
\label{t:t7}
\begin{center}
\begin{tabular}{c|cc|cc}
 & \multicolumn{2}{|c|}{$r_2=r_1$} & \multicolumn{2}{|c}{$r_2=2r_1$}\\
  \scriptsize{ Protocol} &\scriptsize{$\widetilde{T}_{s}, 2v=10^{9}$}& \scriptsize{$\widetilde{T}_{s}, 2v=2\cdot 10^{9}$} & \scriptsize{$\widetilde{T}_{s}, 2v=10^{9}$} & \scriptsize{$\widetilde{T}_{s}, 2v=2 \cdot 10^{9}$}
\\ \hline
NT & 264 & - & 221 & - \\
2C & 310 & 330 & 275 & 302 \\
10T & 329 & - & 179 & - \\
5T2C5T & 397 & 455 & 182 & 186 \\
2C10T & 427 & 448 & 241 & 254 \\
1C5T1C5T & 456 & 507 & 239 & 260 \\
5T1C5T1C & 380 & 421 & 181 & 183 \\
10T2C & 334 & 339 & 179 & 179\\
1C10T1C & 425 &  447 &  234 & 244\\
\end{tabular}
\end{center}
\end{center}
\end{table}

In those cases, it is worth considering which treatment (or treatments) is (are) the optimal one(s) among those under consideration.

\begin{figure}[!h]
    \centering
    \includegraphics[width=0.47\linewidth]{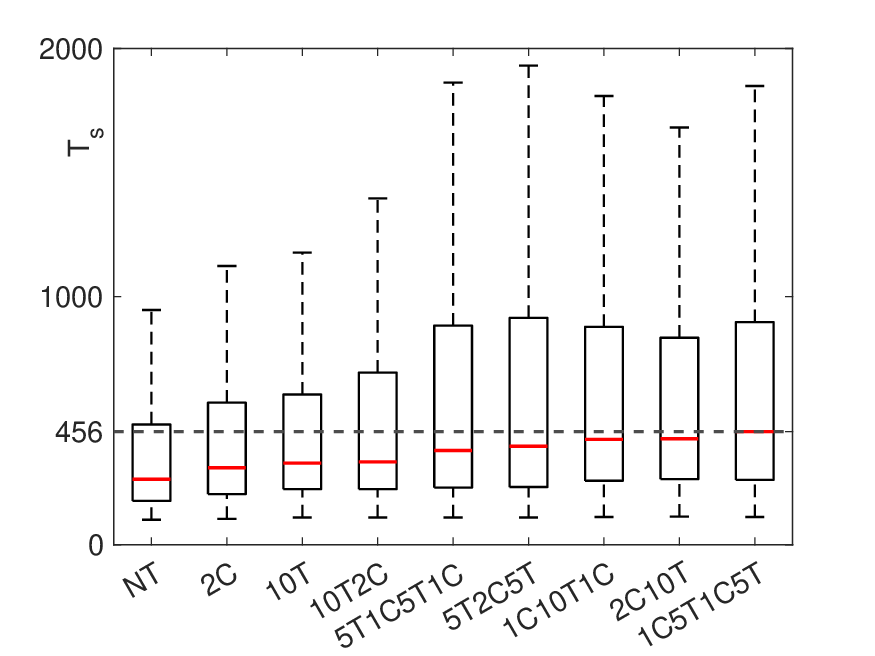}
       \includegraphics[width=0.47\linewidth]{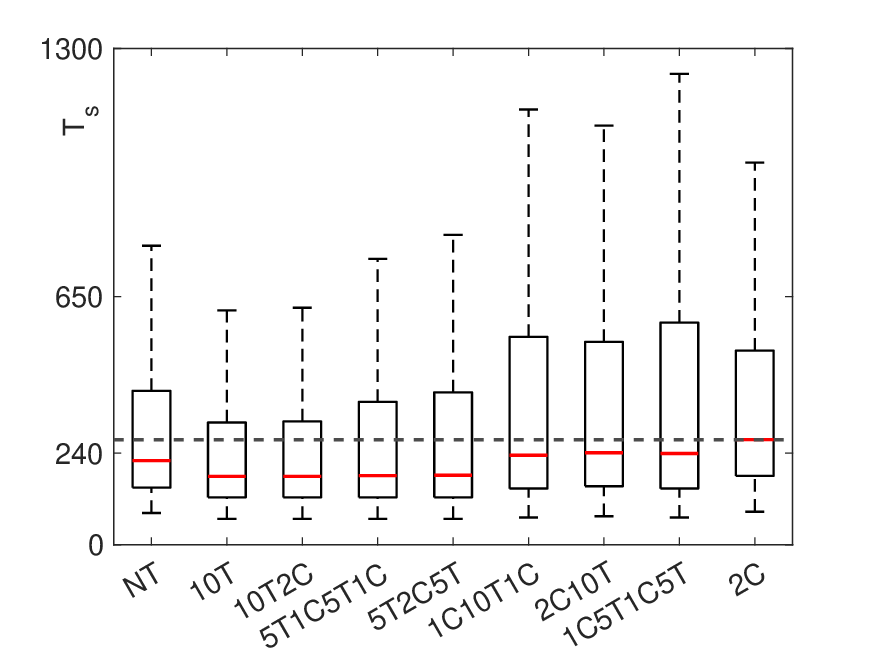}
    \caption{Box plot showing the efficacy of the selected combined protocols for $r_{2}=r_{1}$ (left) and $r_{2}=2r_{1}$ (right) for $2v=10^{9}$.}
    \label{fig:f19}
\end{figure}

We show overall performance of the considered protocols in Table \ref{t:t7} and in Figure \ref{fig:f19}.

\begin{table}[!h]
\caption{The statistically significant correlations ($r$ and p-value) between ${T}_{s}^{1C5T1C5T}$ and model parameters for $r_{2}=r_{1}$, $2v=10^{9}$.}
\begin{tabular}{@{}ccc@{}}
Parameter & r & p-value \\ \hline
$r_{1}$ & -0.66 & 0.00 \\
$\rho_{4}$ & -0.19 & 0.00 \\
$T_{0}$ & -0.15 & 0.00 \\
\label{t:t5c}
\end{tabular}
\end{table}

\begin{table}[!h]
\caption{The statistically significant correlations ($r$ and p-value) between  ${T}_{s}^{1C5T1C5T}/T_{s}^{NT}$ and model parameters for $r_{2}=r_{1}$, $2v=10^{9}$.}
\begin{tabular}{@{}ccc@{}}
Parameter & r & p-value \\ \hline
$\rho_{4}$ & -0.62 & 0.00 \\
$\alpha_{1}$ & 0.31 & 0.00 \\
$\rho_{2}, \rho_{3}$ & 0.26 & 0.00 \\
$r_{1}$ & -0.12 & 0.00 \\
$\delta_{1}$ & -0.11 & 0.00 \\
\label{t:t6c}
\end{tabular}
\end{table}

We find that the best-performing protocol is the 1C5T1C5T protocol with $r_{2} = r_{1}$, such as with $r_{2} = r_{1}/2$. The second best protocol is 2C10T. The correlations between survival time and model parameters for the best-performing protocol, 1C5T1C5T, are presented in Tables \ref{t:t5c} and \ref{t:t6c}. From these tables, we observe that survival time is strongly correlated with tumor growth rate ($r_1$), and weakly correlated with tumor immune suppression ($\rho_4$) and its initial size ($T_0$).

However, when considering the improvement of a combined protocol relative to the absence of treatment, the most important parameter becomes tumor immune suppression ($\rho_4$), with the efficiency of TMZ ($\alpha_1$) and mitotic stimulations ($\rho_2$ and $\rho_3$) also playing significant roles.

The correlations between the model parameters, survival time, and its improvement for the 2C10T and 5T2C5T protocols are very similar to those for 1C5T1C5T, so we do not present them here. Figs. \ref{fig:f20a}, \ref{fig:f20b}, and \ref{fig:f20c} show tumor dynamics for the MVP and KM curves for the 1C5T1C5T, 2C10T, and 5T2C5T protocols, respectively.

\begin{figure}[!h]
    \centering
    \includegraphics[width=0.45\linewidth]{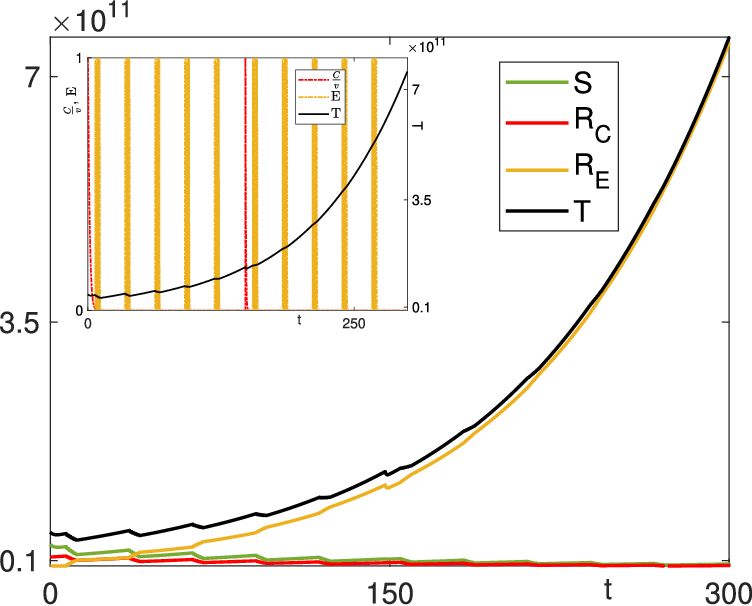}
       \includegraphics[width=0.45\linewidth]{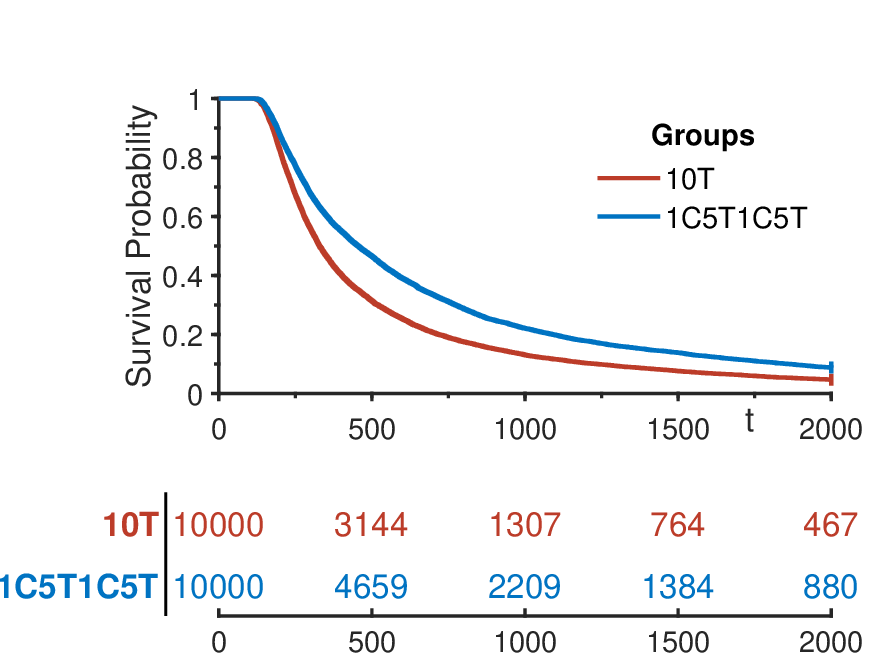}
    \caption{Dynamics for the MVP (left) and KM curves (right) for $r_{2}=r_{1}$ under the 1C5T1C5T protocol, for $v=0.5\cdot 10^{9}$.}
    \label{fig:f20a}
\end{figure}

\begin{figure}[!h]
    \centering
    \includegraphics[width=0.45\linewidth]{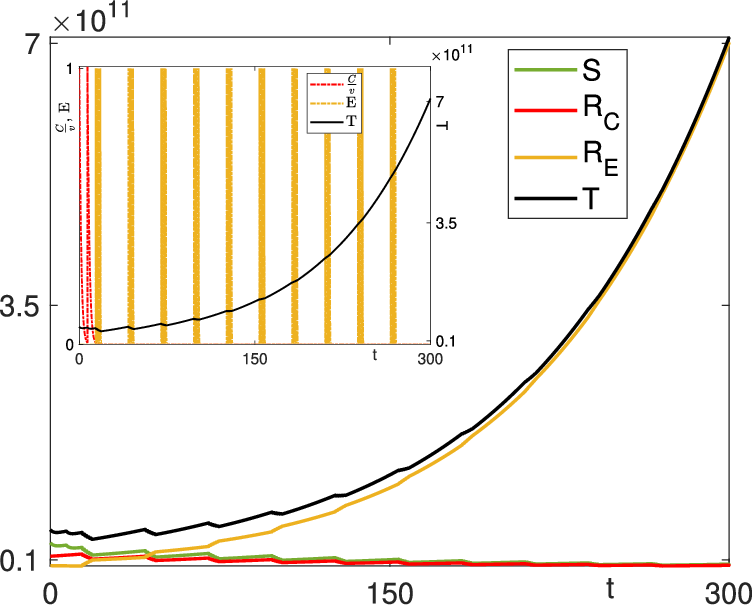}
       \includegraphics[width=0.45\linewidth]{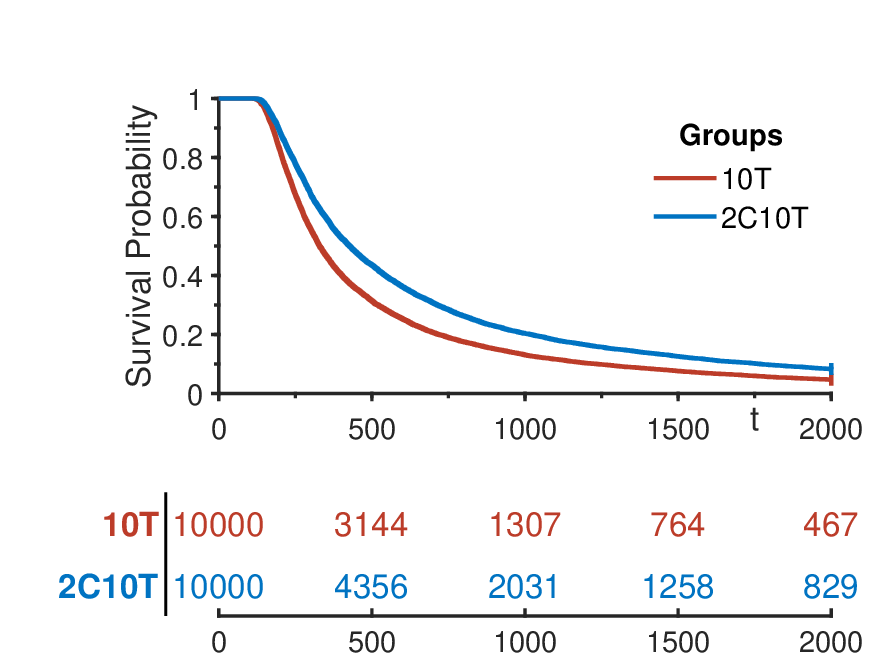}
    \caption{Dynamics for the MVP (left) and KM curve (right) for  $r_{2}=r_{1}$ under the 2C10T protocol, for $v=0.5\cdot 10^{9}$.}
    \label{fig:f20b}
\end{figure}

\begin{figure}[!h]
    \centering
    \includegraphics[width=0.45\linewidth]{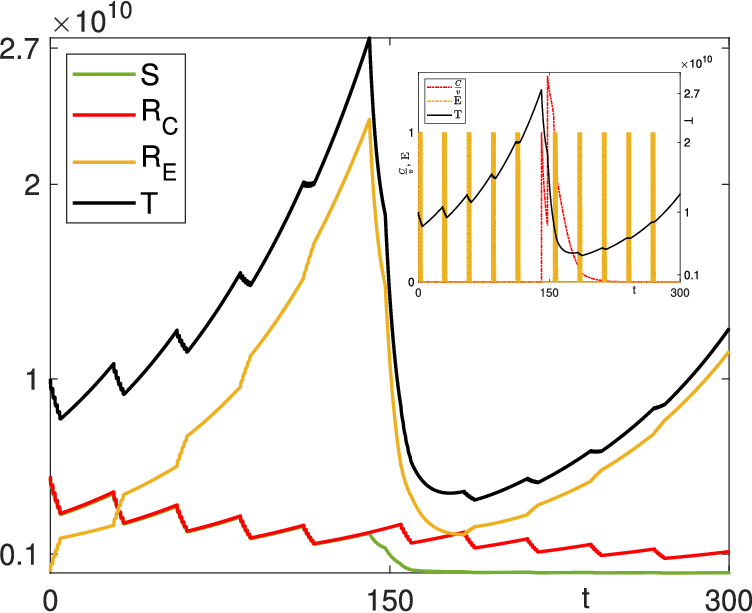}
       \includegraphics[width=0.45\linewidth]{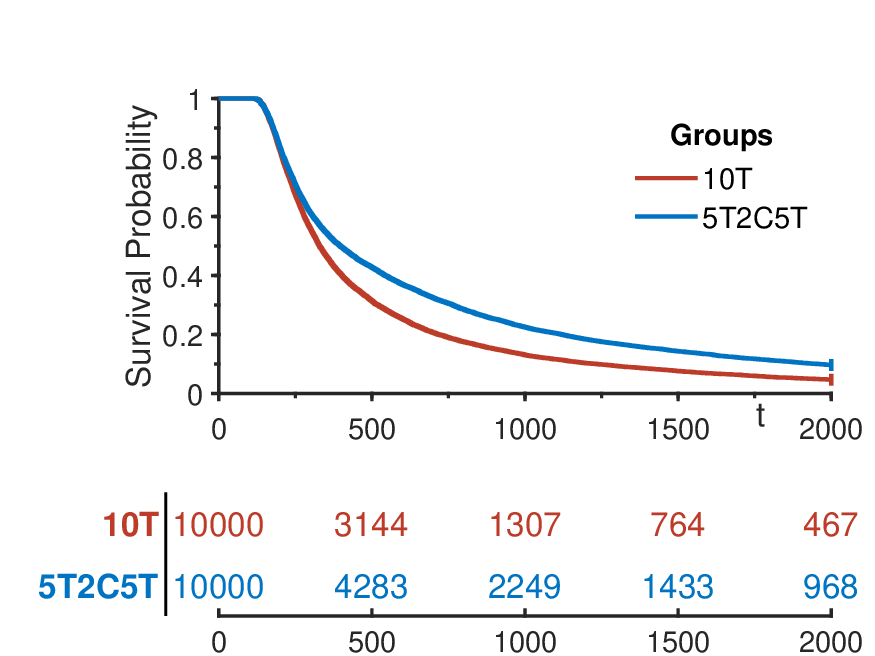}
    \caption{Dynamics for the MVP (left) and KM curve (right) for  $r_{2}=r_{1}$ the 5T2C5T protocol, for $v=0.5\cdot 10^{9}$.}
    \label{fig:f20c}
\end{figure}

\begin{table}[!h]
\caption{The statistically significant correlations ($r$ and p-value) between ${T}_{s}^{2C}$ and model parameters for $r_{2}=2r_{1}$, $2v=10^{9}$.
}
\begin{tabular}{@{}ccc@{}}
Parameter & r & p-value \\ \hline
$r_{1}$ & -0.67 & 0.00 \\
$T_{0}$ & -0.19 & 0.00 \\
$\rho_{4}$ & -0.13 & 0.00 \\
\label{t:t5d}
\end{tabular}
\end{table}

\begin{table}[!h]
\caption{The statistically significant correlations ($r$ and p-value) between  ${T}_{s}^{2C}/T_{s}^{NT}$ and model parameters for $r_{2}=2 r_{1}$, $2v=10^{9}$.}
\begin{tabular}{@{}ccc@{}}
Parameter & r & p-value \\ \hline
$\rho_{4}$ & -0.69 & 0.00 \\
$\delta_{1}$ & -0.26 & 0.00 \\
$\rho_{2}, \rho_{3}$ & 0.23 & 0.00 \\
$T_{0}$ & -0.21 & 0.00 \\
$\delta_{1}$ & -0.16 & 0.00 \\
\label{t:t6d}
\end{tabular}
\end{table}

Interesting, for the case with $r_{2}=2r_{1}$, synergistic effects are not founded in any of the investigated combined protocols. Such as shown in the TMZ monotherapy Section in Fig. \ref{fig:f6}, in the combined protocols, it is better not to administer TMZ at all. In fact, the best protocol reported in Table \ref{t:t7} is the 2C protocol, where only 2 CAR-T injections are administered.

The correlations between survival time and model parameters for the best-performing protocol, 2C, are presented in Table \ref{t:t5d}. We observe that survival time is strongly correlated with tumor growth rate ($r_1$), and weakly correlated with tumor immune suppression ($\rho_4$) and its initial size ($T_0$).

When considering the improvement of the 2C protocol in relation to the absence of treatment reported in Table \ref{t:t6d}, the most important and highly correlated parameter becomes tumor immune suppression ($\rho_4$), with initial percentage of CAR-T-resistant cells ($\delta_{2}$), mitotic stimulations ($\rho_2$ and $\rho_3$), initial tumor size ($T_{0}$), and initial percentage of TMZ-resistant cells ($\delta_{1}$), also playing roles. Note that even here, with only CAR-T monotherapy, the initial percentage of TMZ resistant cells correlates with improvement is survival due to the high aggressiveness ($r_2$) of the TMZ resistant population ($R_E$).

Figure \ref{fig:f20d} shows the tumor dynamics for the MVP and KM curves for the 2C protocol.

\begin{figure}[!h]
    \centering
    \includegraphics[width=0.45\linewidth]{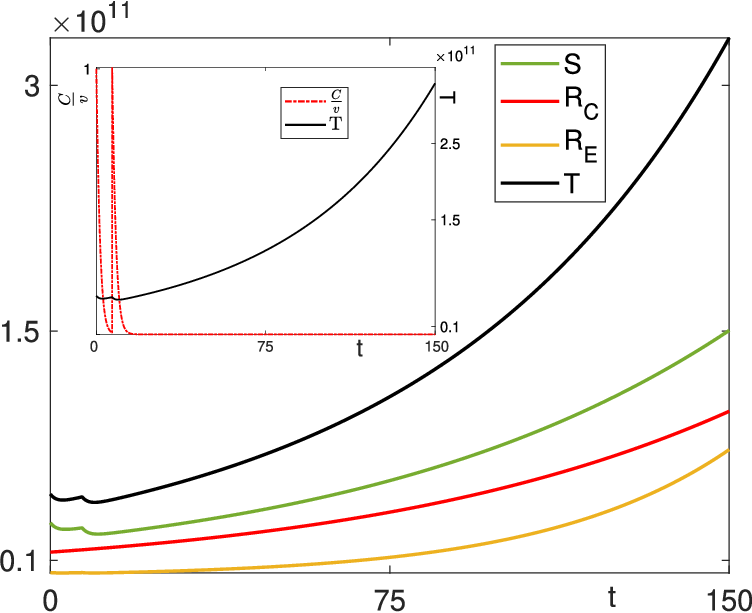}
       \includegraphics[width=0.45\linewidth]{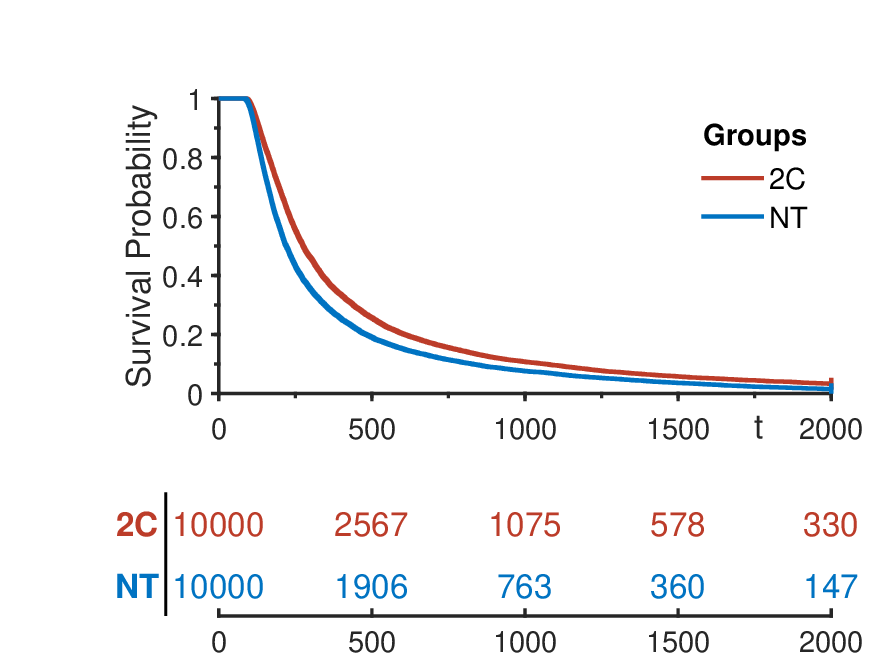}
    \caption{Dynamics for the MVP (left) and KM curve (right) for $r_{2}=2r_{1}$ under the 2C protocol for $v=0.5\cdot 10^{9}$.}
    \label{fig:f20d}
\end{figure}

\section{Discussion}\label{sec:conclusions}

This Section discusses the developed model and its results. We begin with insights from mathematical analysis, followed by simulations findings, which are categorized into monotherapy outcomes, combined therapy outcomes, and fast-growing TMZ-resistant cell cases. Finally, we address future research and general conclusions.

\subsection{The mathematical model}

In this {work}, we propose a mathematical model to investigate potential improvements in glioma treatment by combining two therapies: temozolomide (TMZ) and CAR-T cell therapy. TMZ is a standard treatment for gliomas \cite{friedman2000temozolomide}, often used in multimodal therapies such as the Stupp protocol \cite{stupp_protocl2005}. On the other hand, CAR-T cell therapy is a novel treatment, which has shown promising success in treating non-solid tumors like leukemias \cite{Maude,Miliotou}. Preclinical studies in animal models have demonstrated the potential of combining these therapies \cite{suryadevara2018temozolomide,combined_CART_TMZ}, although their clinical effectiveness is still under investigation in an ongoing trial (NCT04165941) \cite{TMZ_CART_trial}.

Our goal is to contribute to and inform future clinical trials through the insights gained from this work. The strong agreement between simulation results for both TMZ and CAR-T cell monotherapies and real-world clinical findings--discussed below--validates the reliability of the proposed model. This alignment provides confidence in its predictive power and its potential utility in designing and optimizing future combined treatment protocols.

Note that while the primary focus is on gliomas, the model could be adapted and calibrated for other cancers with similar biological mechanisms, offering a basis for similar investigations into combined CAR-T cell therapy and chemotherapy for solid cancers, a new promising and under investigation approach \cite{Moon}.

\subsection{Mathematical Analysis}

{After constructing the model, we analyze it}, initially considering a single dose of both treatments, followed by an analysis with continuous treatment of both therapies.
Interestingly, with a single dose of both treatments, tumor eradication is not possible, which is consistent with clinical observations and the challenges posed by these types of tumors \cite{gliomas2024}.

However, when continuous treatment is applied, we {identify} a threshold for both TMZ and CAR-T cell dosages at which tumor eradication becomes achievable. Importantly, the values presented in Proposition \ref{p:p5} are feasible: the daily critical CAR-T cell dosage $V$ is approximately $5.3\cdot10^{7}$ $\mbox{cells}\cdot \mbox{day}^{-1}$, and when the daily maximal TMZ dosage ($E_0 = 1$) is applied, the critical value for the sum of TMZ efficacy against the tumor and its transition rate from sensitive to TMZ-resistant cells ($\alpha_1 + \epsilon_1$) is approximately 0.21, based on the maximum sensitive growth rate $r_1$ from Table \ref{tab:paramenters}.
Thus, we {find} conditions on the dosages of CAR-T cells and the efficiency and dosage of TMZ that ensure the stability of the equilibrium with zero tumor cells under daily constant treatment.

Continuous treatment is not feasible due to patient toxicity constraints. Therefore, we propose an impulsive treatment strategy, where CAR-T cells and TMZ doses above the critical thresholds are administered {a finite number of times}, driving the system governed by \eqref{eqS}--\eqref{eqE} towards a tumor-free equilibrium, while maintaining treatment within tolerable toxicity levels, thereby ensuring clinical feasibility.

\subsection{TMZ monotherapy}

Regarding TMZ monotherapy applications, our \textit{in silico} trial simulations show that TMZ achieves a double median survival compared to no treatment, gaining almost 8 months (250 days) in median survival. There have been no randomized trials that compared standard chemoradiotherapy with TMZ monotherapy for MG. However, TMZ has been administered alone in some elderly patients. Scott and colleagues conducted a retrospective review of 206 patients and found that patients receiving chemotherapy had a median survival gain of almost 8 months compared to patients not receiving chemotherapy \cite{scott2011aggressive}, as in our \textit{in silico} results.

The relation between the growth rates of sensitive and TMZ-resistant cells is not well characterized \cite{campos2014aberrant,stepanenko2016temozolomide,yuan2018abt,dai2018scd1,gupta2014discordant,stepanenko2016temozolomide,dai2018scd1,Delobel_PLOS}. Importantly, this improvement persists even when TMZ-resistant cells proliferate at rates equal to those of sensitive cells. Administering 10 cycles of TMZ consistently outperforms shorter treatment durations, emphasizing the importance of sustained therapy, as in the Stupp protocol \cite{stupp_protocl2005}--the standard protocol for high-grade gliomas since 2005--where, after 6 weeks of concomitant chemoradiotherapy, adjuvant TMZ cycles are administered until the patient tolerates TMZ toxicity. However, when TMZ-resistant cells proliferate at rates greater than those of sensitive cells, it is more effective to avoid administering TMZ, as virtual patients show no response. This aligns with clinical observations for highly aggressive gliomas, where approximately 50\% of glioblastoma patients do not respond to TMZ treatment \cite{lee2016temozolomide}.

Focusing again on the case with $r_1=0.5r_2$ and the correlation between the model parameters and the survival rate applying 10 TMZ cycles, the parameter showing the highest correlation is the tumor proliferation rate $r_1$, followed by the tumor size $T_0$ and the efficacy of TMZ against the tumor $\alpha_1$. The importance of slow tumor growth as a good prognostic biomarker is well-known: the overexpression of Epidermal Growth Factor Receptor (EGFR) is an indicator of a poor prognosis in overall survival \cite{li2018prognostic}. In addition, the Ki-67 level, a protein marker associated with cell proliferation, serves as a significant prognostic factor in gliomas \cite{chen2015ki}. Higher Ki-67 levels indicate increased tumor cell proliferation, correlating with more aggressive tumor behavior and poorer patient outcomes.
If we now analyze the correlation between model parameters and the improvement in survival compared to untreated cases (the $T_{s}^{10T}/T_{s}^{NT}$ ratio), the most influential parameter is $\alpha_1$. This finding is consistent with \cite{skaga2022functional}, where the authors identified TMZ sensitivity in individual glioblastoma stem cell cultures as a predictor of patient survival.

\subsection{CAR-T monotherapy}

We then analyze the impact of CAR-T cell monotherapy on survival in virtual patients with MG. First, we examine how tumor immunosuppression ($\rho_4$) influences median survival time. Our findings indicate that only patients with low immunosuppression levels in the tumor microenvironment ($\rho_4 \leq 0.1$) derive significant benefit from CAR-T cell therapy. This result aligns with the well-documented challenges posed by the immunosuppressive nature of the solid tumor microenvironment \cite{kringel2023chimeric}.

Next, we show that a higher number of injected CAR-T cells leads to improved survival outcomes, consistent with previous findings \cite{Ode1,Bodnar2023amcs,Forys}. In an experimental study, the authors observed that a low dose of CAR-T cells suppressed GL261/EGFRvIII tumor growth, whereas a high dose resulted in complete eradication of xenograft tumors \cite{chen2019antitumor}. We further investigate the effect of dose distribution on survival and find that single-dose administration at the start is optimal due to the rapid initial expansion of CAR-T cells. However, to mitigate potential adverse and toxicological effects, we adjust the number of CAR-T doses to two ($L_2=2$).

We conduct an \textit{in silico} trial to assess the effect of CAR-T cell monotherapy on median survival, observing an increase of approximately seven months. This finding aligns with a recent review reporting median survival times ranging from 5.5 to 11.1 months across analyzed studies \cite{agosti2024car}.

Next, we analyze the influence of model parameters on patient survival and find that the most correlated parameter is the growth rate $r_1$ of the sensitive tumor population. Interestingly, when considering the correlation between improvement in survival relative to the untreated case--expressed as the $T_{s}^{2C}/T_{s}^{NT}$ ratio--the most influential parameters are those directly related to CAR-T cell treatment: $\rho_4$, $\rho_2$, $\rho_3$, and the fraction of tumor cells resistant to CAR-T therapy, $\delta_1$. These results align with experimental and clinical findings, which identify major challenges in CAR-T therapy for MG as the immunosuppressive tumor microenvironment, insufficient cell trafficking, rapid expansion, and tumor heterogeneity \cite{kringel2023chimeric}.

Indeed, a CAR-T cell product that is resistant to tumor-induced immunosuppression ($\rho_4$) and exhibits rapid expansion towards tumor cells ($\rho_2$ and $\rho_3$) may play a pivotal role in disease control, as already demonstrated in leukemia treatment \cite{Maude,Miliotou}. The efficacy of CAR-T therapy can be further evaluated based on the quality of the patient’s effector cells prior to genetic modification and the CAR-T cell generation employed in the treatment.

This aligns closely with the well-established and critical role of CAR-T cell killing efficiency and rapid stimulation in shaping the therapeutic dynamics of this approach. Indeed, significant efforts have been made by the biomedical community to develop improved CAR-T cell products, characterized by faster expansion, reduced susceptibility to immunosuppression, and enhanced killing efficiency \cite{Sterner}.

Moreover, our model underscores the critical role of tumor-induced immunosuppression, represented by the parameter $\rho_4$, in influencing the behavior of the cell populations studied. The results indicate a clear trend: as the value of $\rho_4$ increases, the effectiveness of the therapy decreases. This finding highlights the potential of incorporating immune checkpoint inhibitors as a strategy to modulate immune responses and prevent tumor cells from deactivating CAR-T cells \cite{xu2020immunotherapy}. By mitigating the tumor-induced inactivation rate of CAR-T cells, this approach could enhance therapeutic outcomes and improve patient prognosis.


\subsection{TMZ and CAR-T combined therapy}

We study the combination of TMZ and CAR-T cell therapies, exploring six different combination protocols to identify the optimal approach to improve survival time. These results are compared with standard 10T treatment to assess potential improvements in therapeutic outcomes.

The best results are obtained with the 5T2C5T and 1C5T1C5T protocols, particularly in the case where the proliferation rate of sensitive cells ($r_1$) is higher than that of resistant cells ($r_2$). Our \textit{in silico} trials predict median survival times of nearly 650 days, representing an improvement of approximately 400, 340, and 100 days compared to no treatment, CAR-T cell monotherapy, and TMZ monotherapy, respectively. Currently, no clinical trial results are available for direct comparison with our \textit{in silico} results.

An experimental study has shown that a combined treatment regimen of MGMT-modified $\gamma\delta$ T cells and TMZ chemotherapy improves survival in \textit{in vivo} primary high-grade gliomas compared to both monotherapy treatments \cite{combined_CART_TMZ}. Moreover, it has been reported that concomitant treatments outperform sequential treatments in terms of survival. Our \textit{in silico} results support the efficacy of combining TMZ with CAR-T cell therapy. Additionally, the superiority of the 1C5T1C5T protocol, where the two treatments alternate and are present simultaneously in virtual patients, aligns with \textit{in vivo} findings \cite{combined_CART_TMZ}.

We further investigate the correlations between model parameters and survival, finding that, even with combined treatment, the biological process most strongly associated with survival is tumor growth. Overexpression of EGFR \cite{li2018prognostic} and Ki-67 levels \cite{chen2015ki} remain reliable prognostic biomarkers in combined treatment applications. Interestingly, when considering survival improvement relative to the untreated case, the most critical parameter is the killing efficacy of TMZ, followed by tumor immunosuppression. This highlights the dominant role of TMZ in combined treatments.

Since combined protocols achieve similar median survival times, we conduct a detailed comparison to identify potential differences. Notably, most virtual patients exhibit similar survival outcomes regardless of the combined protocol used. If virtual patients respond well to one protocol, they are likely to respond similarly to others. However, this preliminary analysis does not account for toxicity, which could vary significantly between protocols and impact clinical decisions.

We then focus on the two best-performing protocols: 5T2C5T and 1C5T1C5T. While their median survival outcomes are nearly identical, a subset of the virtual cohort exhibit significant differences (up to double survival times) when treated with one protocol versus the other. Specifically, the 1C5T1C5T protocol is more effective for patients with shorter survival times, whereas the 5T2C5T protocol shows greater advantages for those with longer survival times.

This observation is further supported by correlations between the ratio $T_s^{1C5T1C5T} / T_s^{5T2C5T}$ and the tumor growth rate $r_1$. Specifically, the 5T2C5T protocol is more effective for tumors with lower $r_1$, while no significant correlation is observed with immunosuppression ($\rho_4$). Conversely, the 1C5T1C5T protocol is more beneficial for tumors with higher $r_1$ and lower $\rho_4$.

These findings underscore the potential and importance of personalized treatment strategies in oncology, highlighting its potential for more precise and effective cancer care. By tailoring therapies to the specific characteristics of each patient's tumor, it becomes possible to optimize treatment efficacy, thereby improving outcomes.

\subsection{Cases with Fast-growing TMZ-resistant Cells}

Finally, given the poorly defined relationship between TMZ-resistant and TMZ-sensitive cell growth in MG  \cite{campos2014aberrant,stepanenko2016temozolomide,yuan2018abt,dai2018scd1,gupta2014discordant,stepanenko2016temozolomide,dai2018scd1,Delobel_PLOS}, we investigate how the efficacy of combined protocols changes with varying levels of aggressiveness of resistant MG. Importantly, we found that the 1C5T1C5T protocol remains the most effective in the case with $r_2=r_1$. This robust finding supports the recommendation of the 1C5T1C5T protocol for future clinical trials, except in cases of highly aggressive TMZ-resistant tumors ($r_2 = 2r_1$).

We also investigate the correlations between survival and model parameters, finding that, even in moderately aggressive TMZ-resistant MG ($r_2=r_1$), the biological process most strongly related to survival is tumor growth. However, when we examine the improvement in terms of survival compared to the untreated case, the most important parameter is tumor immunosuppression, followed by TMZ killing efficacy. These findings underscore the critical role of CAR-T cells in managing aggressive TMZ-resistant populations within combination therapy protocols.

Instead, for highly aggressive TMZ-resistant cells ($r_2=2r_1$), no synergistic benefit from combining TMZ with CAR-T cells. In fact, the best protocol in this scenario is 2C, consistent with the findings from TMZ monotherapy, where TMZ inadvertently favors expansion of the TMZ-resistant population, which can only be controlled through CAR-T cell therapy. This highlights the potential of CAR-T cells in combating TMZ-resistant aggressive MG, such as glioblastomas, where approximately 50\% of patients do not respond to TMZ \cite{lee2016temozolomide}.

In this aggressive scenario, tumor growth of the CAR-T-resistant population is most strongly linked to overall survival. However, when analyzing survival improvements relative to the untreated case, tumor immunosuppression becomes the most influential parameter. Additional contributing factors include the initial fractions of CAR-T- and TMZ-resistant cells, mitotic stimulation, and initial tumor burden. Interestingly, even in protocols involving CAR-T monotherapy, the fraction of TMZ-resistant cells influences survival outcomes due to the high proliferative capacity ($r_2$) of this subpopulation.

Together, these findings highlight the essential role of CAR-T cell therapy in controlling aggressive, TMZ-resistant MG and provide valuable insights for optimizing combination treatment strategies.

\subsection{Future Research}

{Let us remark that for the ranges of parameters given in Table \ref{tab:paramenters} and considered initial conditions we do not observe any complex dynamics in the system \eqref{eqS}--\eqref{eqE}.
This fact may be connected with the finite number of treatment applications and the stability of the invariant hyperplane $K-S_{1}-S_{2}-R=0$ of \eqref{eqS}--\eqref{eqE} at $C=E=0$.
It may also be related to the high growth rates of tumor cells, which, as we have demonstrated above, define the dynamics.
On the other hand, system \eqref{eqS}--\eqref{eqE} may exhibit more complex dynamics for other values of the parameters and/or initial conditions. In this work, our primary goal is to understand how CAR-T therapy and chemotherapy can be effectively combined to maximize median survival in the virtual population of MG virtual patients. On the other hand, it may be an interesting problem to explore possible types of dynamics governed by the system \eqref{eqS}--\eqref{eqE}, which should be considered elsewhere.}

Moreover, it could be essential to consider the role of spatial effects in CAR-T cell therapy, particularly for gliomas, which exhibit infiltrative behavior \cite{owens2024spatiotemporal}. While our model provides valuable insights into the interactions between CAR-T cell therapy, TMZ, and MG cells, spatial dimension is a critical factor in influencing these interactions. To enhance the accuracy and applicability of our findings, future studies should incorporate spatial elements to better assess the efficacy of the proposed treatments in a more realistic setting.

Finally, we aim to extend the proposed mathematical model by incorporating experimental data to more precisely calibrate the parameters, refining our understanding of the treatment dynamics, and evaluating the significance of key model parameters. Calibrating and validating the model on real data would also confirm that all essential biological mechanisms are described within the proposed mathematical model.

\subsection{General Conclusion}

It is important to note that our results are more qualitative than quantitative. Therefore, what we propose is not the exact protocol itself but more the concept of alternating CAR-T injections and TMZ cycles as an optimal synergistic combined treatment strategy. Based on our findings, initiating treatment with a CAR-T injection followed by TMZ administration, and then alternating the two treatments, is suggested as the most promising approach. This strategy is known as evolutionary double-bind \cite{gatenby2009lessons}, and it has already emerged as the optimal administration strategy in previous studies \cite{orlando2012cancer,italia2022optimal}.

Our research paves the way for exploring the synergy between TMZ and CAR-T cell therapy in treating gliomas, a currently incurable tumor. We hope that our findings will encourage further mathematical and medical investigations in this area, contributing to the development of optimized, personalized treatment strategies that combine immunotherapy and chemotherapy for brain tumor patients. Through these efforts, we anticipate making significant strides in improving therapeutic outcomes and ultimately advancing the fight against gliomas.

\section*{Acknowledgments}

This article is part of the research project SBPLY/23/180225/000041, funded by the EU through the ERDF and by the JCCM through INNOCAM, Ministerio de Ciencia e Innovaci\'on, Spain (doi:10.13039/501100011033) and University of Castilla-La Mancha grant 2022-GRIN-34405 (Applied Science Projects within the UCLM research programme). M.I and J.B.-B. were partially supported by European Regional Development Fund (ERDF A way of making Europe) under grant PID2022-142341OB-I00. D.S. is partially supported by H2020-MSCACOFUND-2020-101034228-WOLFRAM2 and PIBA-2024-1-0016.

\section*{Data Availability}

Data sharing is not applicable to this article as no new data were created or analyzed in this study.

\section*{Code Availability}
The simulation codes are available on \href{https://github.com/disinel/-CAR-T-cell-therapy-in-combination-with-chemotherapy-for-malignant-gliomas.git}{GitHub}.

\section*{Appendix}

In this appendix, we show the outcomes for the other combined protocols considered in the article.

  The first protocol that we considered here is the 5T1C5T1C protocol. For this protocol, median survival time is 638 days for $v=10^{9}$ and 673 days for $v=2\cdot 10^{9}$. Then, we have gains of $80$ and $115$ days {(see KM curves  and the dynamics for the MVP in Figs. \ref{fig:f14} and \ref{fig:f15})}.

  \begin{figure}[!h]
    \centering
    \includegraphics[width=0.9\linewidth]{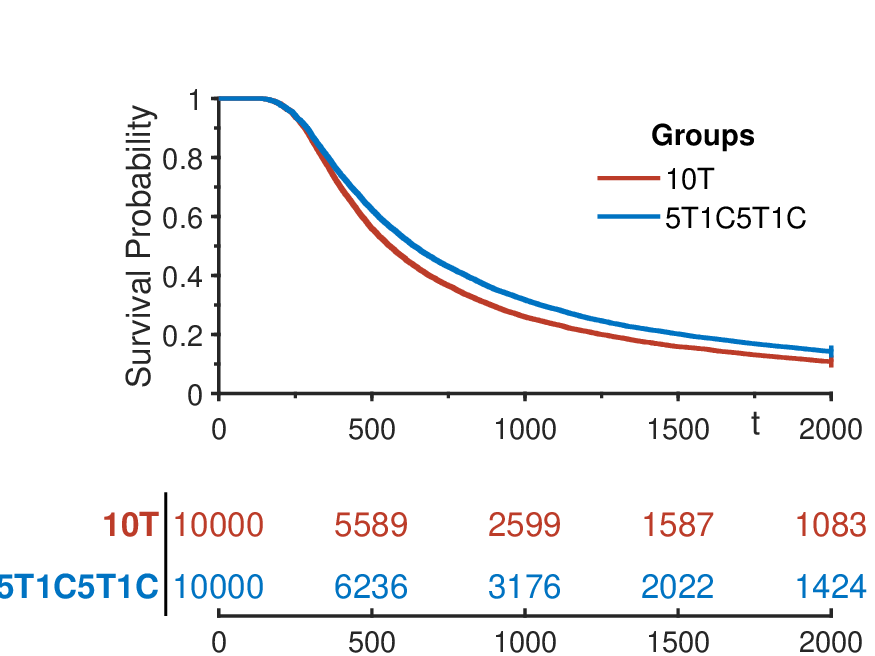}
       \includegraphics[width=0.9\linewidth]{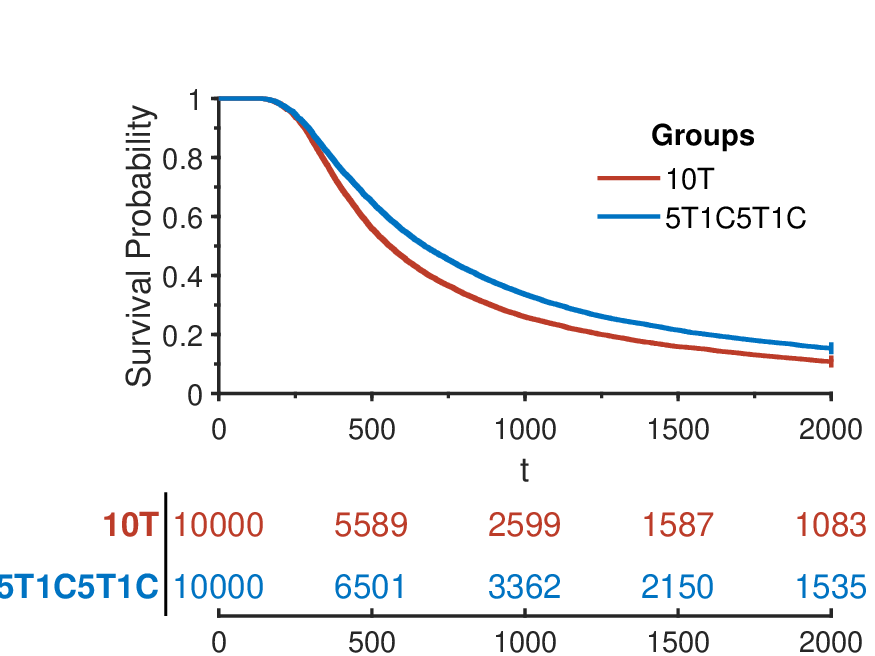}
    \caption{KM curves comparing the 10T and 5T1C5T1C protocols for $v=10^{9}$ (left) and $v=2\cdot 10^{9}$ (right)}
    \label{fig:f14}
\end{figure}

\begin{figure}[!h]
    \centering
    \includegraphics[width=0.9\linewidth]{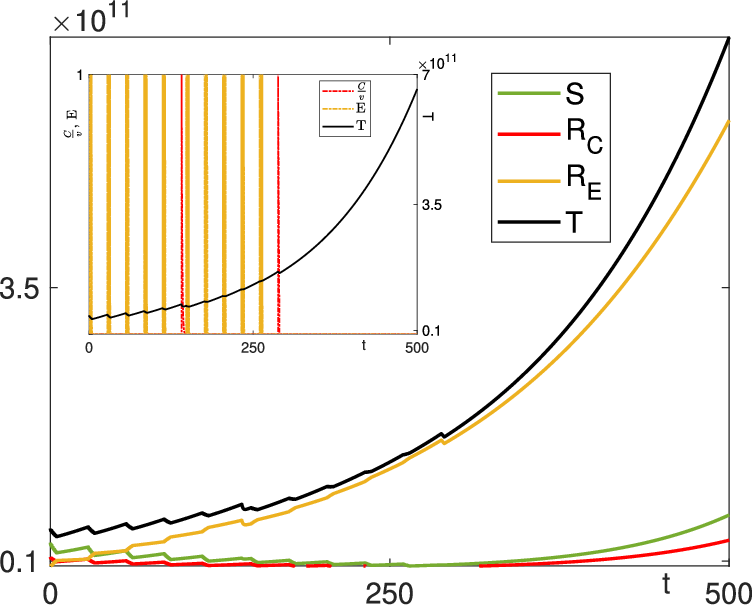}
    \caption{Dynamics for the MVP under the 5T1C5T1C protocol for $v=10^{9}$ }
    \label{fig:f15}
\end{figure}

\begin{figure}[!h]
    \centering
    \includegraphics[width=0.9\linewidth]{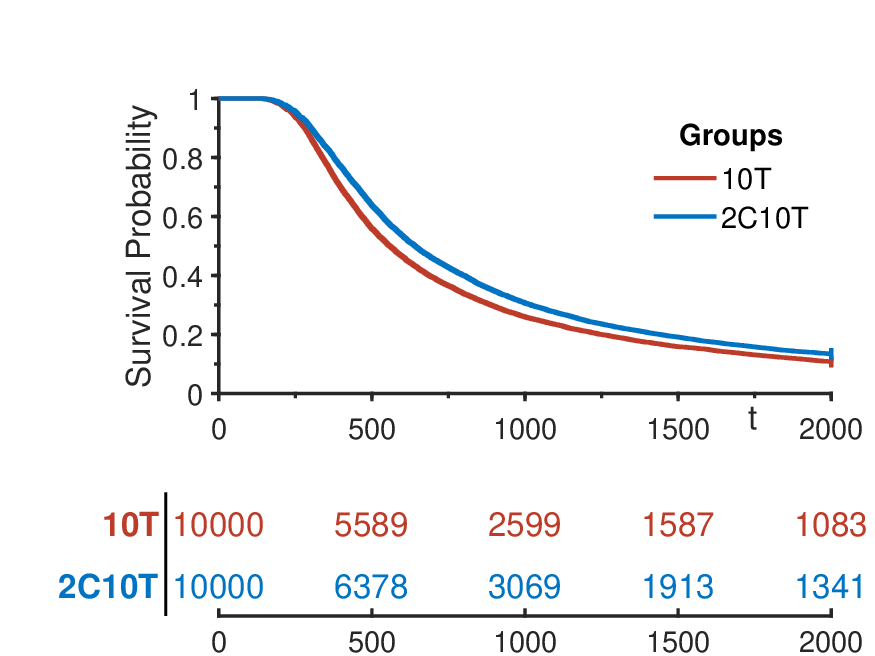}
       \includegraphics[width=0.9\linewidth]{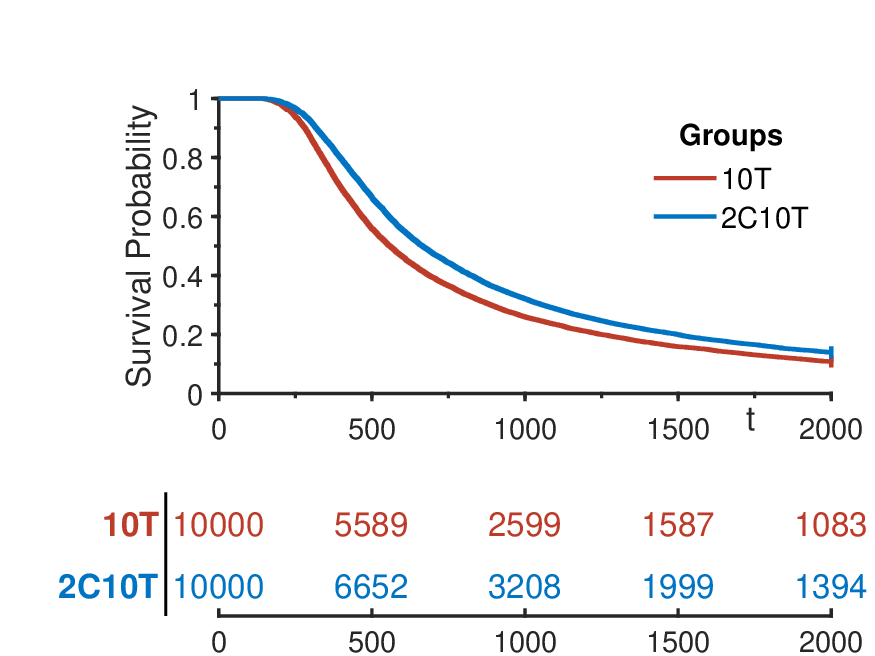}
    \caption{KM curves comparing the 10T and the 2C10T protocols for $2v=10^{9}$ (top) and $2v=2\cdot 10^{9}$ (bottom).}
    \label{fig:f10}
\end{figure}

\begin{figure}[!h]
    \centering
    \includegraphics[width=0.9\linewidth]{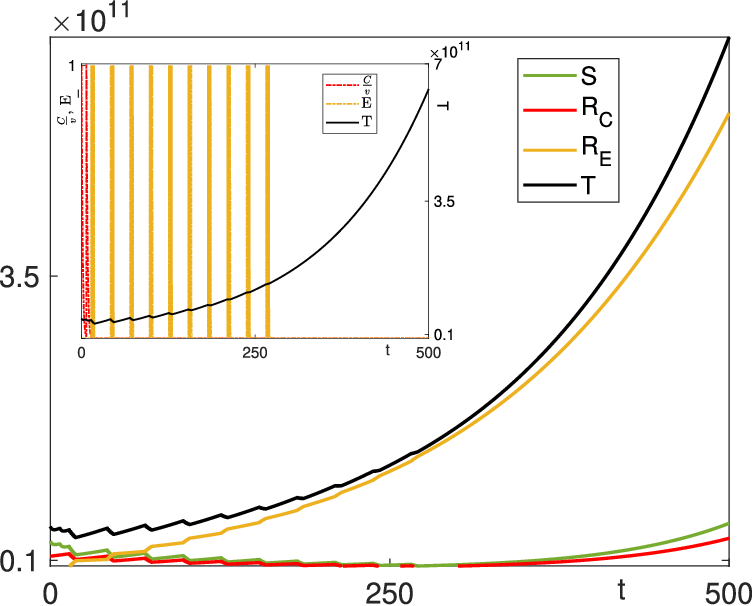}
    \caption{Dynamics for the MVP under 2C10T protocol for $2v=10^{9}$. }
    \label{fig:f11}
\end{figure}

Next, we consider the 2C10T protocol. Its median survival time is 641 days for $L_2v=10^{9}$ and 665 days for $L_2v=2\cdot 10^{9}$. Thus, we have gains of 14.87\% and 19.18\%, i.e., $83$ and $107$ days, respectively (see KM curves in Fig. \ref{fig:f10}). We also present the dynamics of tumor cells for this protocol applied to the MVP {(Fig. \ref{fig:f11})}. This protocol is less efficient in dealing with TMZ-resistant cells.

\begin{figure}[!h]
    \centering
    \includegraphics[width=0.9\linewidth]{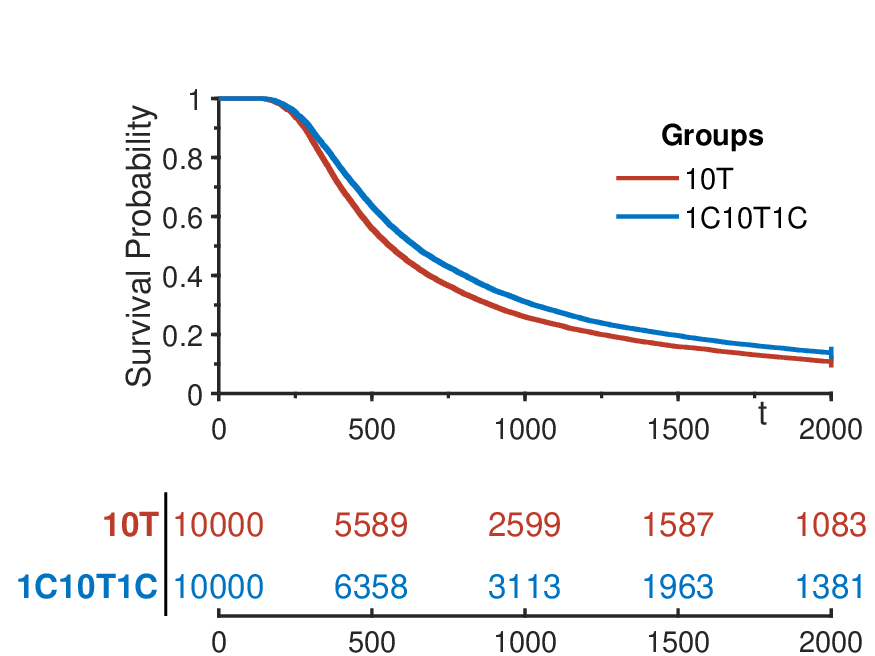}
       \includegraphics[width=0.9\linewidth]{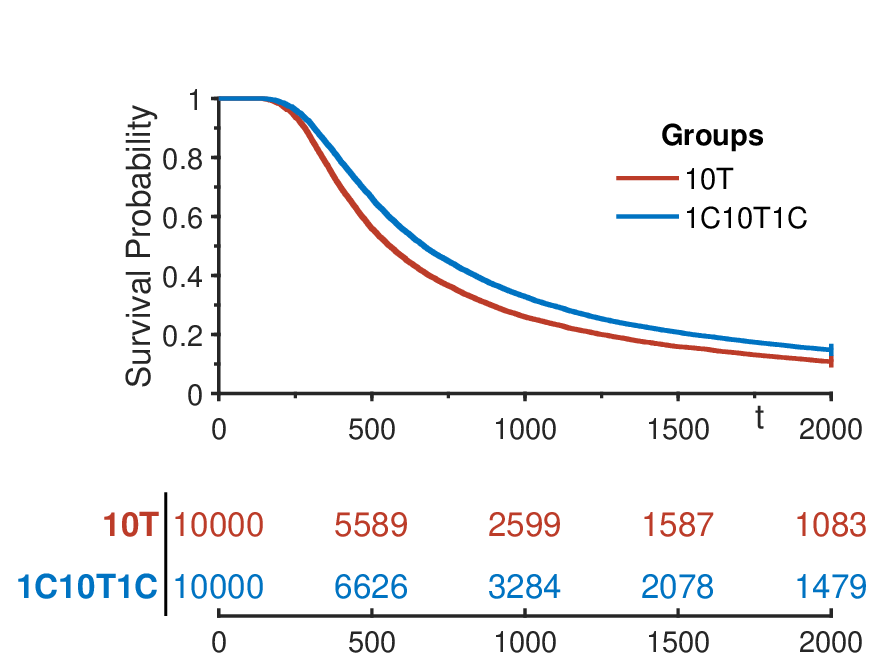}
    \caption{KM curves comparing the 10T and 1C10T1C protocols for $2v=10^{9}$ (top) and $2v=2\cdot 10^{9}$ (bottom).}
    \label{fig:f10a}
\end{figure}

\begin{figure}[!h]
    \centering
    \includegraphics[width=0.9\linewidth]{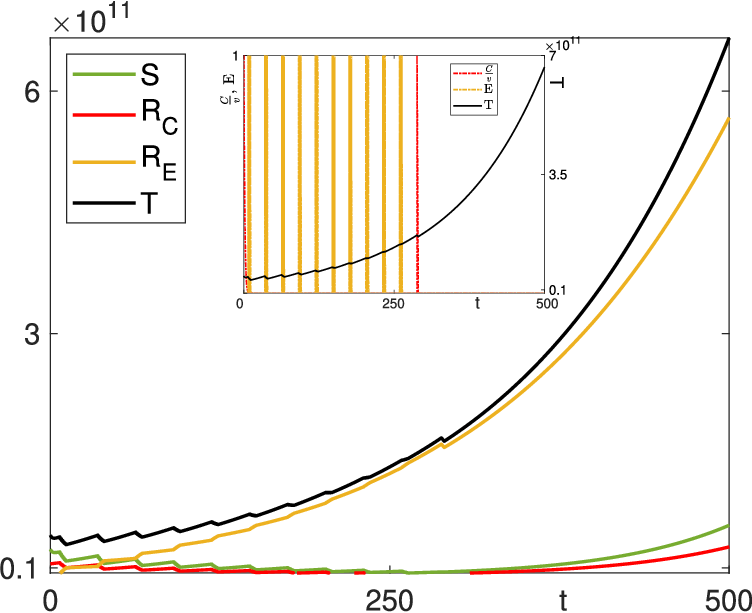}
    \caption{Dynamics for the MVP under the 1C10T1C protocol for $2v=10^{9}$. }
    \label{fig:f11a}
\end{figure}

Another protocol that we consider is the 1C10T1C protocol. Its median survival is 641 and 670 days for $L_2v=10^{9}$ and $L_2v=2\cdot 10^{9}$, respectively. Thus, we have gains of 14.87\% and 20.00\%, i.e., $83$ and $112$ days, respectively (see KM curves in Fig. \ref{fig:f10a}). We also present the dynamics of the model variables for this protocol applied to the MVP in {Fig. \ref{fig:f11a}}: this protocol is less efficient in dealing with TMZ-resistant cells.

Finally, we consider the protocol 10T2C. Median survival time for this protocol is 603 days for $v=10^{9}$ and 630 days for $v=2\cdot 10^{9}$. Thus, we have gains of $45$ and $72$ days {(see Figs. \ref{fig:f16} and \ref{fig:f17})}.

\begin{figure}[!h]
    \centering
    \includegraphics[width=0.9\linewidth]{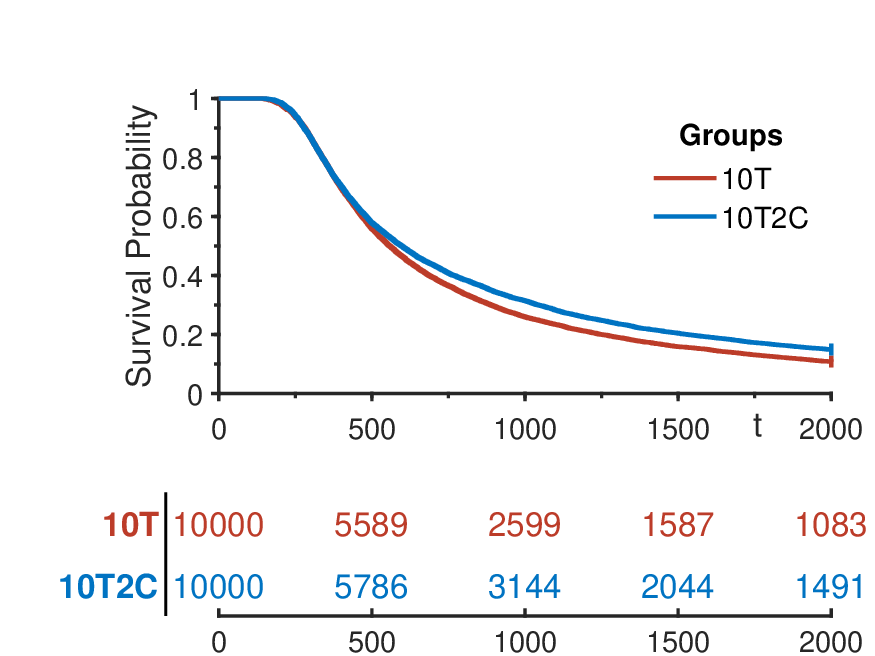}
       \includegraphics[width=0.9\linewidth]{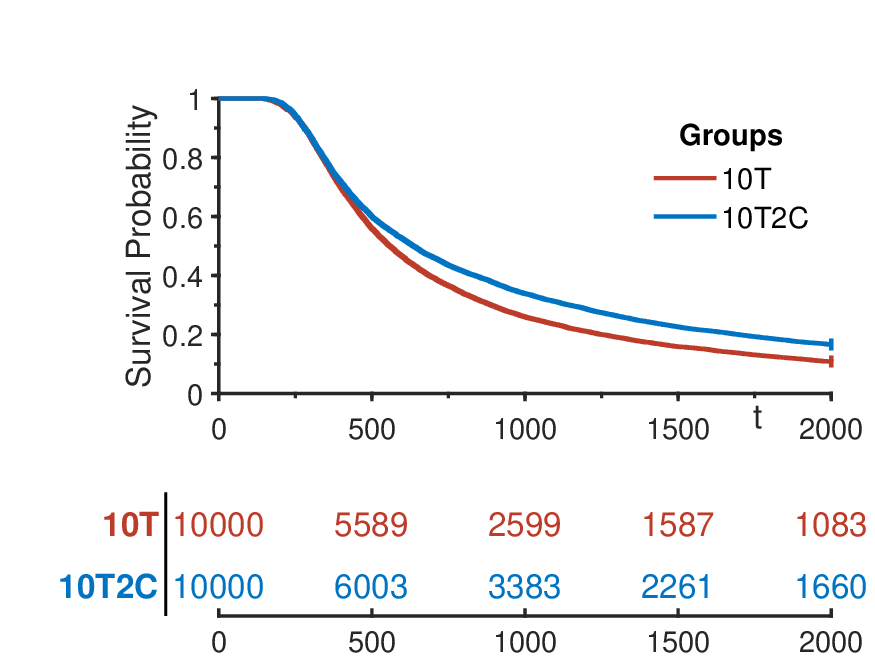}
    \caption{KM curves comparing the 10T and 10T2C protocols for $2v=10^{9}$ (left), $2v=2\cdot 10^{9}$ (right)}
    \label{fig:f16}
\end{figure}

\begin{figure}[!h]
    \centering
    \includegraphics[width=0.9\linewidth]{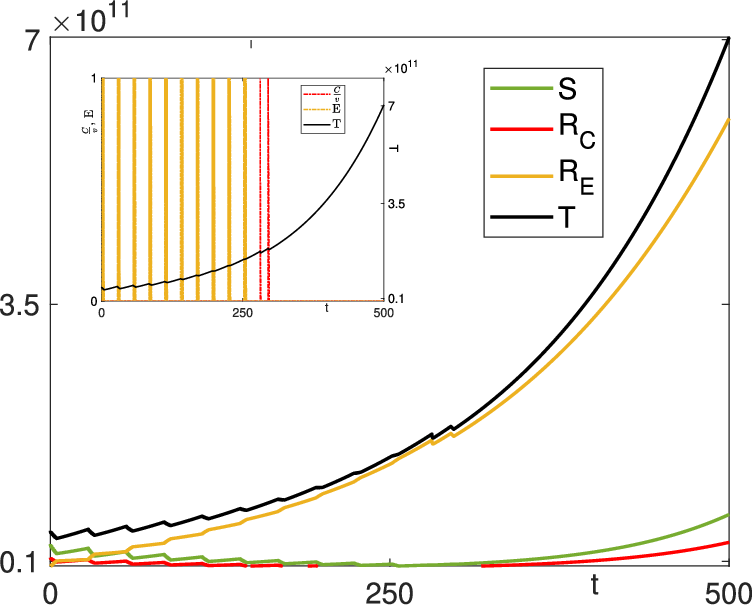}
    \caption{Dynamics for the MVP under the 10T2C protocol for $2v=10^{9}$. }
    \label{fig:f17}
\end{figure}

\section*{References}
\bibliography{bibliog}

\end{document}